\documentclass[12pt]{amsart}
\usepackage{txfonts}      
\usepackage{amssymb}
\usepackage{eucal}
\usepackage{amsmath}
\usepackage{amscd}
\usepackage{xcolor}
\usepackage{multicol}
\usepackage[all]{xy}           
\usepackage{graphicx}
\usepackage{color}
\usepackage{colordvi}
\usepackage{xspace}
\usepackage{amssymb}
\usepackage{tikz}
\usepackage{makecell}
\usepackage{appendix}
\usepackage{enumerate,enumitem}
\usepackage{amsthm}
\usepackage[thicklines]{cancel}
\usepackage[compress]{cite}
\usepackage{ifpdf}
\ifpdf
\usepackage[colorlinks,final,backref=page,hyperindex]{hyperref}
\else
\usepackage[colorlinks,final,backref=page,hyperindex,hypertex]{hyperref}
\fi

\usepackage[active]{srcltx} 
\newtheorem{thm}{Theorem}[section]
\newtheorem{lem}[thm]{Lemma}
\newtheorem{cor}[thm]{Corollary}
\newtheorem{pro}[thm]{Proposition}
\theoremstyle{definition}
\newtheorem{defi}[thm]{Definition}
\newtheorem{ex}[thm]{Example}
\newtheorem{rmk}[thm]{Remark}

\title[ Noncommutative pre-Poisson bialgebras and relative Rota-Baxter operators] {Noncommutativepre-Poisson bialgebras and relative Rota-Baxter operators}

\author{Hongliang Li}
\address{Department of Mathematics, Zhejiang International Studies University, Hangzhou, 310023} \email{honglli@126.com}

\author{Qinxiu Sun$^*$}
\address{Department of Mathematics, Zhejiang University of Science and Technology, Hangzhou, 310023} \email{qxsun@126.com}


\begin{document}
\begin{abstract}

In this paper, we develop the bialgebra theory for coherent noncommutative pre-Poisson algebras and 
establish equivalences among matched pairs, Manin triples, the phase space of noncommutative Poisson algebras
 and noncommutative pre-Poisson bialgebras.
The investigation of coboundary noncommutative pre-Poisson bialgebras naturally leads to
 the noncommutative pre-Poisson Yang-Baxter equation (NPP-YBE).
We prove that a symmetric solution of the NPP-YBE gives rise to a (coboundary) noncommutative 
pre-Poisson bialgebra. Moreover, we demonstrate how solutions without the symmetry 
condition can also generate such bialgebras. This motivates the introduction of quasi-triangular and 
factorizable noncommutative pre-Poisson bialgebras.
 In particular, we show that a solution of the NPP-YBE with an invariant skew-symmetric 
 part yields a quasi-triangular noncommutative pre-Poisson bialgebra. 
 Such solutions are further interpreted as relative Rota-Baxter operators with weights.
Finally, we establish a one-to-one correspondence between quadratic Rota-Baxter noncommutative pre-Poisson 
algebras and factorizable noncommutative pre-Poisson bialgebras.

\end{abstract}

\maketitle

\vspace{-1.2cm}

\tableofcontents

\vspace{-1.2cm}

\allowdisplaybreaks

\section{Introduction}

The purpose of this paper is to develop a bialgebra theory for noncommutative pre-Poisson algebras.
The study of coboundary noncommutative pre-Poisson bialgebras naturally gives rise to
 the noncommutative pre-Poisson Yang-Baxter equation (NPP-YBE).
 In particular, we investigate quasi-triangular and factorizable noncommutative pre-Poisson bialgebras,
  which constitute a subclass of coboundary noncommutative pre-Poisson bialgebras. 
Furthermore, we explicitly construct a correspondence linking factorizable noncommutative 
pre-Poisson bialgebras to quadratic Rota-Baxter noncommutative pre-Poisson algebras.  

\subsection{(Noncommutative) Poisson algebras and (noncommutative) pre-Poisson algebras}
A Poisson algebra is an algebraic structure that combines a Lie algebra with a
 commutative associative algebra, linked together by the Leibniz rule. 
 Poisson algebras originated from the study of Hamiltonian mechanics 
 and have since become important in many areas of mathematics and mathematical physics.
  These include Poisson geometry \cite{Va,W}, 
   algebraic geometry \cite{Gk,Po}, 
quantum mechanics \cite{Di}, quantum groups \cite{Dr} and quantization theory \cite{Hu,Ko}. 

Aguiar integrated the structures of Zinbiel algebras \cite{Lo1} and pre-Lie algebras \cite{B3} on the same vector space, 
thereby introducing the notion of a pre-Poisson algebra \cite{A1}. Zinbiel algebras are
Koszul dual to Leibniz algebras, and they are special
dendriform algebras. The anti-commutator of a Zinbiel algebra gives a commutative associative
algebra, and the commutator of a pre-Lie algebra gives a Lie algebra. A pre-Poisson algebra induces a 
Poisson algebra naturally
through the sub-adjacent commutative associative algebra of the Zinbiel algebra and
the sub-adjacent Lie algebra of the pre-Lie algebra. Conversely, a Rota-Baxter operator
 (more generally an $\mathcal{O}$-operator action) on a Poisson algebra gives 
rise to a pre-Poisson algebra. We can summarize these relations by the following diagram:
$$
\xymatrix{
\ar[rr] \mbox{{\bf Zinbiel algebra} + pre-Lie algebra }\ar[d]_{\mbox{sub-adjacent~~}}
                && \mbox{Pre-Poisson algebra}\ar[d]_{\mbox{sub-adjacent~~}}\\
\ar[rr] \mbox{{\bf comm associative algebra} + Lie algebra }\ar@<-1ex>[u]_{\mbox{~~Rota-Baxter ~operator}}
                && \mbox{ Poisson algebra. }\ar@<-1ex>[u]_{\mbox{~~Rota-Baxter~ operator}}}
$$

A noncommutative Poisson algebra, first formally introduced by Xu \cite{Xu}, provides a natural 
framework for various geometric contexts. It extends the classical 
notion of a Poisson algebra by replacing the underlying commutative associative algebra with a general 
(not necessarily commutative) associative algebra, while retaining a compatible Lie algebra
structure linked via the Leibniz rule. This relaxation of commutativity results in significantly 
richer and more complex algebraic behaviors. Noncommutative Poisson algebras have been studied by many authors
from different aspects, see \cite{Ku1, Ku2,Ku3,Yy2,Lbs} and reference therein.
 
\begin{defi}
A {\bf noncommutative Poisson algebra} is a triple $(A,\cdot, [ \ , \ ])$ such that $(A,\cdot)$ is an associative algebra
 (not necessarily commutative), $(A, [ \ , \ ])$ is a Lie algebra and the following Leibniz rule holds:
\begin{equation}\label{Np0}[x,y\cdot z]=[x,y]\cdot z+y\cdot[x,z],\quad \forall~ x,y,z\in A.\end{equation}
It is called {\bf coherent} if, in addition,
\begin{equation}\label{Np00}[x,y\cdot z]+[y,z\cdot x]+[z,x\cdot y]=0,\quad \forall~ x,y,z\in A.\end{equation}
\end{defi}

Since Zinbiel algebras can be viewed as commutative dendriform algebras \cite{Lo2}, 
a natural generalization is to combine dendriform algebras with pre-Lie algebras.
 This leads to the notion of a noncommutative pre-Poisson algebra \cite{Lbs}. 
 The correspondence between Poisson algebras and pre-Poisson algebras is formally analogous to
  that between their noncommutative versions, a relationship illustrated in the following diagram,
$$
\xymatrix{
\mbox{{\bf dendriform algebra} + pre-Lie algebra }\ar[rr] \ar[d]_{\mbox{sub-adjacent~~}}
                && \mbox{noncomm pre-Poisson algebra}\ar[d]_{\mbox{sub-adjacent~~}}\\
\mbox{{\bf associative algebra} + Lie algebra }\ar[rr] \ar@<-1ex>[u]_{\mbox{~~Rota-Baxter ~operator}}
                && \mbox{noncomm Poisson algebra}\ar@<-1ex>[u]_{\mbox{~~Rota-Baxter ~operator}}}
$$

\subsection{Quasi-triangular bialgebras }
A bialgebraic structure comprises an algebra and a coalgebra equipped with certain compatibility conditions. 
 In the early 1980s, Drinfeld established the theory of Lie bialgebras \cite{Dr1},
  which were found to have deep connections with the classical Yang-Baxter equation and classical integrable systems. 
  Subsequently, V. Zhelyabin introduced associative D-bialgebras in \cite {Z1,Z2}. 
  Aguiar later in \cite{A2} studied antisymmetric infinitesimal bialgebras as the associative analogue of Lie bialgebras, 
  which are also equivalent to double constructions of Frobenius algebras \cite{B2}. 
  Following the development of infinitesimal bialgebra theory \cite{A2,B2}, 
  similar bialgebraic frameworks have been extended to various other algebraic structures, 
  such as pre-Lie algebras \cite{B1}, Leibniz algebras \cite{Ts} and
   Novikov algebras \cite{Hbg}.
A Manin triple of Poisson algebras corresponds to a Poisson bialgebra \cite{Nb},
 offering a natural framework for constructing compatible Poisson brackets in integrable systems.
 Unlike in the commutative case, a direct generalization of Poisson bialgebra theory 
 is not straightforward due to fundamental differences in representation theory. In \cite{Lbs}, the authors
 developed the noncommutative Poisson bialgebra theory 
 using quasi-representations and the corresponding cohomology theory \cite{Yy1,Yy2}.
 
 Within the theory of Lie bialgebras, coboundary Lie bialgebras especially quasi-triangular Lie bialgebras,
 which play a fundamental role in mathematical physics. As a specialized subclass of quasi-triangular Lie bialgebras,
  factorizable Lie bialgebras \cite{Ls} provide a crucial connection between classical $r$-matrices and certain factorization problems.
   They exhibit diverse applications in integrable systems, see \cite{Rs, St} and references therein.
  Recently, these results on factorizable and quasi-triangular structures have been successfully 
   extended to antisymmetric infinitesimal bialgebras \cite{Sw}, pre-Lie bialgebras \cite{Wbl} and Leibniz bialgebras \cite{Bls}.
   
 \subsection{ Outline of the paper}

 A bialgebra theory has been developed for the usual (commutative) pre-Poisson algebras,
 so-called pre-Poisson bialgebras \cite{Zls}, in terms of the representation theory of pre-Poisson algebras.
 However, a direct generalization of this framework does not apply to noncommutative pre-Poisson algebras. 
 This limitation arises because the dual of the regular representation of a noncommutative Poisson algebra 
 is generally not a representation, but only a quasi-representation. To overcome this obstacle, 
 we develop a bialgebra theory specifically for coherent noncommutative pre-Poisson algebras. 
 Since every pre-Poisson algebra is in particular a coherent noncommutative pre-Poisson algebra, 
 such a generalization is both natural and well justified.

This paper is organized as follows. In Section 2, we introduce quasi-representations and matched pairs
 for noncommutative pre-Poisson algebras. The notion of a phase space of 
 a noncommutative Poisson algebra is also studied, which later be used to characterize the bialgebra theory.
  Section 3 is devoted to establishing a bialgebra theory for noncommutative pre-Poisson algebras. 
  Investigating the coboundary case leads to the noncommutative pre-Poisson Yang-Baxter equation (NPP-YBE), 
  whose symmetric solutions give rise to noncommutative pre-Poisson bialgebras.
   We further define $\mathcal{O}$-operators on noncommutative pre-Poisson algebras and
    use them to construct symmetric solutions of the NPP-YBE. 
    In Section 4, we explore quasi-triangular and factorizable noncommutative pre-Poisson bialgebras.
     We prove that the double of any noncommutative pre-Poisson bialgebra 
     naturally carries a factorizable structure. Finally, Section 5 presents the notion of quadratic 
     Rota-Baxter noncommutative pre-Poisson algebras and develops a one-to-one correspondence 
     between these and factorizable noncommutative pre-Poisson bialgebras.

 {\bf Notations.} Throughout the paper, $k$ is a field.  All vector spaces and algebras are over $k$. 
 All algebras are finite-dimensional, although many results still hold in the infinite-dimensional case.
 Let $V$ be a vector space with a binary operation $\ast$. Define linear maps
$L_{\ast}, R_{\ast},\mathrm{ad}:V\rightarrow \hbox{End}(V)$ by
 $L_{\ast}(a)b:=a\ast b, \  \ R_{\ast}(a)b:=b\ast a, \ \ \mathrm{ad}(a)b=a\ast b-b\ast a$~ for all$~a, b\in V$.
Assume that
 $r=\sum\limits_{i}a_i\otimes b_i \in V\otimes V$. Put
 \begin{small}
\begin{align*}
r_{12}\ast r_{13}:=\sum_{i,j}a_i\ast a_j\otimes b_i\otimes b_j,\;r_{23}\ast r_{12}:=\sum_{i,j}a_j\otimes a_i\ast b_j\otimes b_i,\;
r_{31}\ast r_{23}:=\sum_{i,j}b_i\otimes a_j\otimes a_i\ast b_j,\\
r_{21}\ast r_{13}:=\sum_{i,j}b_i\ast a_j\otimes a_i\otimes b_j,\;
r_{32}\ast r_{21}:=\sum_{i,j}b_j\otimes b_i\ast a_j\otimes a_i,\;
r_{31}\ast r_{32}:=\sum_{i,j}b_i\otimes b_j\otimes a_i\ast a_j,\\
r_{13}\ast r_{32}:=\sum_{i,j}a_i\otimes b_j\otimes b_i\ast a_j,\;
r_{23}\ast r_{21}:=\sum_{i,j}b_j\otimes a_i\ast a_j\otimes b_i,\;
r_{21}\ast r_{31}:=\sum_{i,j}b_i\ast b_j\otimes a_i\otimes a_j,\\
r_{23}\ast r_{13}:=\sum_{i,j}a_i \otimes a_j \otimes b_i\ast b_j,\;
r_{12}\ast r_{31}:=\sum_{i,j}a_i\ast b_j\otimes b_i\otimes a_j.
\end{align*}\end{small}

\section{(Quasi-) representations and phase spaces of noncommutative Poisson algebras}
In this section, we introduce the notions of (quasi-) representations, matched pairs
 and Manin triples for noncommutative pre-Poisson algebras. 
 We also consider the phase space of noncommutative Poisson algebras.
  Finally, we exhibit the relationship between matched pairs,
   Manin triples of noncommutative pre-Poisson algebras and the phase space of noncommutative Poisson algebras.
We begin by recalling the definitions of pre-Lie algebras \cite{B3}, 
dendriform algebras \cite{Lo2} and noncommutative pre-Poisson algebras \cite{Lbs}.

\subsection{(Quasi-) representations of noncommutative pre-Poisson algebras}
Let $A$ be a vector space with a bilinear product $\ast:A\otimes
A\longrightarrow A$, $(A,\ast )$ is called a pre-Lie algebra if
\begin{equation*}(x\ast y)\ast z-x\ast(y\ast z)=(y\ast x)\ast z-y\ast(x\ast
z),~\forall~x, y, z \in A.\end{equation*}

\begin{defi} A {\bf dendriform algebra} is a vector space $A$ together with two bilinear maps 
$\succ,\prec:A\otimes A\longrightarrow A$ such that the following equalities hold for all $x,y,z\in A$:
  \begin{align*}&
    (x\prec y)\prec z=x\prec (y\succ z+y\prec z),\\&
    (x\succ y)\prec z=x\succ(y\prec z),\\&
    x\succ(y\succ z)=(x\succ y+x\prec y)\succ z.
  \end{align*}
  \end{defi}

\begin{defi}
  A {\bf noncommutative pre-Poisson algebra} is a quadruple $(A,\succ,\prec,\ast)$ such that $(A,\succ,\prec)$ is a dendriform algebra and $(A,\ast)$ is a pre-Lie algebra satisfying the following compatibility conditions:
  \begin{eqnarray}
    \label{Np1}(x\ast y-y\ast x)\succ z&=&x\ast(y\succ z)-y\succ(x\ast z),\\
    \label{Np2} x\prec(y\ast z-z\ast y)&=&y\ast(x\prec z)-(y\ast x)\prec z,\\
    \label{Np3} (x\succ y+x\prec y)\ast z&=&(x\ast z)\prec y+x\succ(y\ast z).
  \end{eqnarray}
  A noncommutative pre-Poisson algebra $(A,\succ,\prec,\ast)$  is called {\bf coherent} if it also satisfies
  \begin{equation}\label{Np4}
    (x\succ y+x\prec y)\ast z=x\ast(y\succ z)+y\ast(z\prec x).
  \end{equation}
\end{defi}

Let $(A,\succ,\prec,\ast)$ be a (coherent) noncommutative pre-Poisson algebra.
Define
  	\begin{equation*}x\cdot y=x\succ y+x\prec y, \ \ \ [x,y]=x\ast y-y\ast x,~\forall ~x,y\in A.\end{equation*} 
  Then $(A,\cdot, [ \ , \ ])$ is a (coherent) noncommutative Poisson algebra, which is called the 
  sub-adjacent noncommutative Poisson algebra of $(A,\succ,\prec,\ast)$.
  
\begin{ex} \label{Ae}  Let $A$ be a 2-dimensional vector space with a basis $\{e_1,e_2\}$.
Define bilinear maps $\succ,\prec,\ast:A\otimes A \longrightarrow A$ respectively by
 (only non-zero multiplications are listed):
\begin{align*} &e_1\succ e_2=e_2,\ \ \ e_1 \succ e_1=e_1, \ \ \ e_1 \prec e_2=-e_2,
 \\& e_2\prec e_1=e_2 , \ \ \  e_1\ast e_1=e_1, \ \ \ e_2\ast e_1=e_2. \end{align*}  
  By a direct computation, $(A,\succ,\prec,\ast)$ is a coherent noncommutative pre-Poisson algebra.
\end{ex}

\begin{ex} \label{Ae1}
Let $A$ be a 3-dimensional vector space with a basis $\{e_1,e_2,e_3\}$.
Define binary products $\succ,\prec,\ast:A\otimes A \longrightarrow A$ respectively by
 (only non-zero multiplications are listed):
 \begin{align*}
e_2\succ e_2=e_3, \ \ \  e_2\prec e_2=-e_3,\ \ \
e_1\ast e_2=e_3, \ \ \ e_2\ast e_2=e_1.
	\end{align*}
By a direct computation, $(A,\succ,\prec,\ast)$ is a coherent noncommutative pre-Poisson algebra.
\end{ex}

We now turn to the study of quasi-representations and representations of noncommutative pre-Poisson algebras.

\begin{defi} 
Let $(A,\succ,\prec,\ast)$ be a noncommutative pre-Poisson algebra,
$V$ a vector space and $l_{\succ},r_{\succ},l_{\prec},r_{\prec},l_{\ast},r_{\ast}: A
		\longrightarrow \hbox{End} (V)$ be linear maps. 
\begin{enumerate}
	\item
		$(V,l_{\succ},r_{\succ},l_{\prec},r_{\prec},l_{\ast},r_{\ast})$ is called a
		{\bf quasi-representation} of $(A,\succ,\prec,\ast)$ if $(V,l_{\succ},r_{\succ},l_{\prec},r_{\prec})$ is
		a representation of the dendriform algebra $(A,\succ,\prec)$, $(V,l_{\ast},r_{\ast})$ is a representation 
of the pre-Lie algebra
		$(A,\ast)$ and they satisfy the following conditions for all $x,y\in A$:
		\begin{align}&\label{r1}l_{\succ}(x\ast y-y\ast x)=l_{\ast}(x)l_{\succ}(y)-l_{\succ}(y)l_{\ast}(x),\\&
		\label{r2}r_{\succ}(y)(l_{\ast}(x)-r_{\ast}(x))=l_{\ast}(x)r_{\succ}(y)-r_{\succ}(x\ast y),\\&
		\label{r3}l_{\prec}(x)(l_{\ast}(y)-r_{\ast}(y))=l_{\ast}(y)l_{\prec}(x)-l_{\prec}(y\ast x),\\&
		\label{r4}r_{\prec}(y\ast x-x\ast y)=l_{\ast}(y)r_{\prec}(x)-r_{\prec}(x)l_{\ast}(y),
	\\&\label{r5}r_{\ast}(y)(l_{\succ}(x)+l_{\prec}(x))=l_{\prec}(x\ast y)+l_{\succ}(x)r_{\ast}(y),
\\&\label{r6}r_{\ast}(x)(r_{\succ}(y)+r_{\prec}(y))=r_{\prec}(y)r_{\ast} (x)+r_{\succ}(y\ast x).\end{align}
	\item  A quasi-representation $(V,l_{\succ},r_{\succ},l_{\prec},r_{\prec},l_{\ast},r_{\ast})$ is called a
{\bf representation} of $(A,\succ,\prec,\ast)$ if the following conditions are satisfied:
\begin{align}&\label{r7}r_{\succ}(x)(l_{\ast}(y)-r_{\ast}(y))=l_{\succ}(y)r_{\ast}(x)-r_{\ast}(y\succ x),\\&
		\label{r8}l_{\prec}(x)(l_{\ast}(y)-r_{\ast}(y))=r_{\prec}(y)r_{\ast}(x)-r_{\ast}(x\prec y),\\&
		\label{r9}l_{\ast}(x\succ y+x\prec y)=r_{\prec}(y)l_{\ast}(x)+l_{\succ}(x)l_{\ast}(y).\end{align}
\end{enumerate}
	 \end{defi}
By Eqs.~(\ref{r7})-(\ref{r9}), we have
\begin{align}&\label{r70}r_{\cdot}(x)(l_{\ast}(y)-r_{\ast}(y))
+l_{\cdot}(y)(l_{\ast}(x)-r_{\ast}(x))=l_{\ast}(y\cdot x)-r_{\ast}(y\cdot x),
\\&\label{r80}r_{\cdot}(x)(l_{\ast}(y)-r_{\ast}(y))+(l_{\ast}(y)-r_{\ast}(y))r_{\cdot}(x)=r_{\cdot}(y\ast x-x\ast y).\end{align}

Let $A$ and $V$ be vector spaces. For
a linear map $f: A \longrightarrow \hbox{End} (V)$, define a linear
map $f^{*}: A \longrightarrow \hbox{End} (V^{*})$ by $\langle
f^{*}(x)u^{*},v\rangle=\langle u^{*},f(x)v\rangle$ for all $x\in A,
u^{*}\in V^{*}, v\in V$, where $\langle \ , \ \rangle$ is the usual
pairing between $V$ and $V^{*}$.

\begin{pro} \label{Dr} 
   Let $(A,\succ,\prec,\ast)$ be a noncommutative pre-Poisson algebra.
   \begin{enumerate}
 \item If $(V,l_{\succ},r_{\succ},l_{\prec},r_{\prec},l_{\ast},r_{\ast})$ is a quasi-representation of $(A,\succ,\prec,\ast)$, 
 then $(V^*,r_{\cdot}^{*},-l_{\prec}^{*},-r_{\succ}^{*},$ $l_{\cdot}^{*},r_{\ast}^{*}-l_{\ast}^{*},r_{\ast}^{*})$ is also a quasi-representation of $(A,\succ,\prec,\ast)$.

 \item If $(V,l_{\succ},r_{\succ},l_{\prec},r_{\prec},l_{\ast},r_{\ast})$ is a quasi-representation and satisfies
    \begin{align}&\label{r10}l_{\ast}(x)(x\succ y+x\prec y)=l_{\ast}(x)l_{\succ}(y)+l_{\ast}(y)r_{\prec}(x),\\&
		\label{r11}r_{\ast}(y)(l_{\succ}(x)+l_{\prec}(x))=l_{\ast}(x)r_{\succ}(y)+
r_{\ast} (y\prec x),\\&
		\label{r12}r_{\ast}(x)(r_{\succ}(y)+r_{\prec}(y))=r_{\ast} (y\succ x)+l_{\ast}(y)l_{\prec}(x)),\end{align}
   then $(V^*,r_{\cdot}^{*},-l_{\prec}^{*},-r_{\succ}^{*},l_{\cdot}^{*},r_{\ast}^{*}-l_{\ast}^{*},r_{\ast}^{*})$ is a representation of $(A,\succ,\prec,\ast)$.
    \item If $(V,l_{\succ},r_{\succ},l_{\prec},r_{\prec},l_{\ast},r_{\ast})$ is a quasi-representation of $(A,\succ,\prec,\ast)$,
    then both $(V,l_{\succ},r_{\prec},l_{\ast})$ 
   and $(V,l_{\cdot},r_{\cdot},r_{\ast}-l_{\ast},r_{\ast})$
     are quasi-representations of the sub-adjacent 
    noncommutative Poisson algebra $(A,\cdot,[ \ , \ ])$,
    where $r_{\cdot}=r_{\succ}+r_{\prec},~l_{\cdot}=l_{\succ}+l_{\prec}$.
   \end{enumerate}
\end{pro}

\begin{proof} (a) By Proposition 3.2.2 \cite{B2} and Proposition 3.3 \cite{B1}, 
we know that $(V^*,r_{\cdot}^{*},-l_{\prec}^{*},-r_{\succ}^{*},l_{\cdot}^{*})$ is a
representation of $(A,\succ,\prec)$ and
$(V^{*},r_{\ast}^{*}-l_{\ast}^{*},r_{\ast}^{*})$ is a
representation of $(A,\ast)$. It therefore remains to verify that Eqs.~(\ref{r1})–(\ref{r6}) hold for the tuple
 $(V^*,r_{\cdot}^{*},-l_{\prec}^{*},-r_{\succ}^{*},l_{\cdot}^{*},r_{\ast}^{*}-l_{\ast}^{*},r_{\ast}^{*})$.
Indeed, Eq.~(\ref{r1}) follows from Eqs.~(\ref{r1}), (\ref{r4}) and (\ref{r6}); Eq.~(\ref{r2}) follows from Eq.~(\ref{r3});
Eq.~(\ref{r3}) follows from Eq.~(\ref{r2}); Eq.~(\ref{r4}) follows from Eqs.~(\ref{r1}) and (\ref{r3});
Eq.~(\ref{r6}) follows from Eq.~(\ref{r5}); Eq.~(\ref{r5}) follows from Eq.~(\ref{r6}).

(b) By Eq.~(\ref{r12}), we get that Eq.~(\ref{r7}) holds; by Eq.~(\ref{r11}), we get that Eq.~(\ref{r8}) holds;
by Eqs.~(\ref{r11}) and (\ref{r12}), we get that Eq.~(\ref{r9}) holds.

(c) It can routinely be checked.
\end{proof}

\begin{ex} \label {Cr} Let $(A,\succ,\prec,\ast)$ be a noncommutative pre-Poisson algebra.
Then  \begin{enumerate}
 \item $(A,L_{\succ}, R_{\succ},L_{\prec}, R_{\prec},L_{\ast}, R_{\ast})$ is a
representation of $(A,\succ,\prec,\ast)$, which is called the {\bf regular
representation} of $(A,\succ,\prec,\ast)$.
 \item $(A^{*},R_{\cdot}^{*},-L_{\prec}^{*},-R_{\succ}^{*},L_{\cdot}^{*},
-\mathrm{ad}^{*},R^{*}_{\ast})$ is a 
quasi-representation of $(A,\succ,\prec,\ast)$.
 \item $(A^{*},R_{\cdot}^{*},-L_{\prec}^{*},-R_{\succ}^{*},L_{\cdot}^{*},
 -\mathrm{ad}^{*},R^{*}_{\ast})$ is a
representation of $(A,\succ,\prec,\ast)$ if and only if $(A,\succ,\prec,\ast)$ is coherent.
 \item $(A, L_{\succ},R_{\prec},L_{\ast})$ is a representation of the sub-adjacent 
    noncommutative Poisson algebra $(A,\cdot,[ \ , \ ])$
 and $(A^{*},R_{\prec}^{*}, L_{\succ}^{*},-L_{\ast}^{*})$ is a quasi-representations of $(A,\cdot,[ \ , \ ])$.
\item  
 $(A^{*},R_{\prec}^{*}, L_{\succ}^{*},-L_{\ast}^{*})$ is a representation of the sub-adjacent 
    noncommutative Poisson algebra $(A,\cdot,[ \ , \ ])$ if  $(A,\succ,\prec,\ast)$ is coherent.
 \end{enumerate}
\end{ex}

Let us recall the notion of matched pairs of noncommutative Poisson algebras, as introduced in \cite{Lbs}.

\begin{pro} \label{Df}
Let $(P_1,\cdot_1,[ \ , \ ]_{1})$ and $(P_2,\cdot_2,[ \ , \ ]_{2})$ be two noncommutative Poisson algebras. Assume that 
 there are linear maps $l_{1},r_{1},\rho_1:P_1\longrightarrow \hbox{End} (P_2)$ and $l_{2},r_{2},\rho_2:P_2\longrightarrow\hbox{End} (P_1)$.
 such that $(P_1,P_2,l_{1},r_{1},l_{2},r_{2})$ is a matched pair of associative algebras and
  $(P_1,P_2,\rho_1,\rho_2)$ is a matched pair of Lie algebras. Moreover,
  the following equalities hold for all $x,y\in P_1,a,b\in P_2$, 
\begin{eqnarray}&&
\label{MP1} \rho_1(x)(a\cdot_2 b)=(\rho_{1}(x)a)\cdot_2 b+a\cdot_2 \rho_{1}(x)b-l_{1}(\rho_2(a)x)b-r_{1}(\rho_2(b)x)a,\\&&
 \label{MP2} l_{1}(x)[a,b]_{2}=[a,l_{1}(x)b]_{2}-\rho_{1}(r_{2}(b) x)a-l_{1}(\rho_2(a)x)b+(\rho_{1}(x)a)\cdot_2 b,\\&&
\label{MP3}  \rho_2(a)(x\cdot_1 y)=(\rho_2(a)x)\cdot_1 y+x\cdot_1 \rho_2(a)y-l_{2}(\rho_{1}(x)a)y-r_{2}(\rho_{1}(y)a)x,\\&&
\label{MP4} l_{2}(a)[x,y]_{1}=[x,l_{2}(a) y]_{1}-\rho_2(r_{1}(y)a)x-l_{2}(\rho_1(x)a)y+(\rho_2(a)x)\cdot_1 y.
\end{eqnarray}
Define two binary operations $\cdot$ and $[ \ , \ ]$ on $P_1\oplus P_2$ by
\begin{align*}&(x+a)\cdot(y+b)=x\cdot_{1}y+l_{\cdot_2}(a)y+r_{\cdot_2}(b)x+a\cdot_{2}b+l_{\cdot_1}(x)b+r_{\cdot_1}(y)a,\\&
[x+a,y+b]=[x,y]_{1}+\rho_2(a)y-\rho_2(b)x+[a,b]_{2}+\rho_1(x)b-\rho_1(y)a.
\end{align*}
Then $(P_1\oplus P_2,\cdot,[ \ , \ ])$ is a noncommutative Poisson algebra.
Denote this noncommutative Poisson algebra by $A_1\bowtie A_2$ and
$(P_{1},P_{2},l_{1},r_{1},\rho_1,l_{2},r_{2},\rho_2)$
satisfying the above conditions is called a {\bf matched pair of
noncommutative Poisson algebras}. Conversely, any noncommutative Poisson algebra that can be decomposed into a
direct sum of two noncommutative Poisson subalgebras is obtained from a matched pair of noncommutative Poisson algebras.
\end{pro}

\begin{rmk} \label{La} \begin{enumerate}
\item If the linear maps
 $l_{2},r_{2},\rho_2: P_{2}\rightarrow \text{End}(P_1)$ are trivial, then the matched pair $(P_{1},P_{2},l_{1},r_{1},\rho_1,l_{2},r_{2},\rho_2)$
reduces to an A-noncommutative Poisson algebra. 
For the completeness, we list it in detail. Let $(P_1,\cdot_1,[ \ , \ ]_{1})$ and
$(P_2,\cdot_2,[ \ , \ ]_{2})$ be two noncommutative Poisson algebras. Assume that there are linear maps
$l_{1},r_{1},\rho_1:P_1\longrightarrow \hbox{End}(P_2)$
 such that
$(P_2,l_{1},r_{1},\rho_1)$ is a representation of  $(P_1,\cdot_1,[ \ , \ ]_{1})$ and
the following compatible conditions hold for all $x,y\in P_1,a,b\in P_2$:
\begin{small}
	\begin{align*}&l_1(x)(a\cdot_2 b)=(l_{1}(x)a)\cdot_2 b, \ \ \  \ \ r_1(x)(a\cdot_2 b)=a\cdot_2 (r_{1}(x)b),
\\&(r_{1}(x)a)\cdot_2 b=a\cdot_2 (l_{1}(x)b), \ \ \ \rho_1(x)[a,b]_2=[\rho_1(x)a,b]_2+[a,\rho_1(x)b]_2,
	\\& \rho_1(x)(a\cdot_2 b)=(\rho_{1}(x)a)\cdot_2 b+a\cdot_2 \rho_{1}(x)b,\ \ \ 
 l_{1}(x)[a,b]_{2}=[a,l_{1}(x)b]_{2}+(\rho_{1}(x)a)\cdot_2 b.
\end{align*}
\end{small}
 Then $(P_{2},l_{1},r_{1},\rho_1)$ is called
  an A-noncommutative Poisson algebra.
  \item If the linear maps
 $l_{2},r_{2},\rho_2: P_{2}\rightarrow \text{End}(P_1)$ and $(\cdot_2,[ \ , \ ]_{2})$ are trivial, 
 then the matched pair 
reduces to the semidirect product. Denote it simply by $A_{1}\ltimes A_2$.
  \end{enumerate}
\end{rmk}

\begin{pro}\label{Mp}
	 Let $(A_1,\succ_1,\prec_1,\ast_1)$ and $(A_2,\succ_2,\prec_2,\ast_2)$ be two noncommutative pre-Poisson algebras. Suppose that there are linear maps
$l_{\succ_1},r_{\succ_1},l_{\prec_1},r_{\prec_1},l_{\ast_1},r_{\ast_1}:A_1\longrightarrow
\hbox{End} (A_2)$ and
$l_{\succ_2},r_{\succ_2},l_{\prec_2},r_{\prec_2},l_{\ast_2},r_{\ast_2}:A_2\longrightarrow
\hbox{End} (A_1)$.
Define operations $\succ,\prec,\ast$ on $A_1\oplus A_2$ by
\begin{align}
&\label{ppmp eq1.1}(x+a)\succ
(y+b)=x{\succ_1}y+l_{\succ_2}(a)y+r_{\succ_2}(b)x+a{\succ_2}b+l_{\succ_1}(x)b+r_{\succ_1}(y)a,\\
&\label{ppmp eq1.2}(x+a)\prec
(y+b)=x{\prec_1}y+l_{\prec_2}(a)y+r_{\prec_2}(b)x+a{\prec_2}b+l_{\prec_1}(x)b+r_{\prec_1}(y)a,\\
&\label{ppmp eq1.3}(x+a)\ast
(y+b)=x{\ast_1}y+l_{\ast_2}(a)y+r_{\ast_2}(b)x+a{\ast_2}b+l_{\ast_1}(x)b+r_{\ast_1}(y)a,
\end{align}
for all $~x,y\in A_1,a,b\in A_2$.
Then $(A_{1}\oplus A_{2},\succ,\prec,\ast)$ is a noncommutative pre-Poisson algebra 
if and only if  the following conditions hold:
\begin{enumerate}
	\item $(A_2,l_{\succ_1},r_{\succ_1},l_{\prec_1},r_{\prec_1},l_{\ast_1},r_{\ast_1})$ is a representation of $(A_1,\succ_1,\prec_1,\ast_1)$.
	\item $(A_1,l_{\succ_2},r_{\succ_2},l_{\prec_2},r_{\prec_2},l_{\ast_2},r_{\ast_2})$ is a representation of $(A_2,\succ_2,\prec_2,\ast_2)$.
	\item $(A_1,A_2, l_{\succ_1},r_{\succ_1},l_{\prec_1},r_{\prec_1},l_{\succ_2},r_{\succ_2},l_{\prec_2},r_{\prec_2})$ 
is a matched pair of dendriform algebras.
	\item $(A_1,A_2,l_{\ast_1},r_{\ast_1},l_{\ast_2},r_{\ast_2})$ is a matched pair of pre-Lie algebras.
	\item The following compatible conditions hold for all $x,y\in A_1,a,b\in A_2$:
\begin{small}
	\begin{align}&\label{ppmp eq1.4}r_{\succ_2}(a)[x, y]_1=x\ast_{1}
		r_{\succ_2}(a)y+r_{\ast_2}(l_{\succ_1}(y)a)x-y\succ_{1}(r_{\ast_2}(a)x)-r_{\succ_2}(l_{\ast_1}(x)a)y,
\\&\label{ppmp eq1.5}
(r_{\ast_2}(a)x-l_{\ast_2}(a)x)\succ_1 y+l_{\succ_2}(l_{\ast_1}(x)a-r_{\ast_1}(x)a)y=x\ast_1(l_{\succ_2}(a)y)+
r_{\ast_2}(r_{\succ_1}(y)a)x-l_{\succ_2}(a)(x\ast_1 y),
\\&\label{ppmp eq1.6}(l_{\ast_2}(a)x-r_{\ast_2}(a)x)\succ_1 y-l_{\succ_2}(r_{\ast_1}(x)a-l_{\ast_1}(x)a)y=
l_{\ast_2}(a)(x\succ_1 y)-x\succ_1(l_{\ast_2}(a)y)
-r_{\succ_2}(r_{\ast_1}(y)a)x,
\\&\label{ppmp eq1.7}x\prec_1(r_{\ast_2}(a)y-l_{\ast_2}(a)y)+r_{\prec_2}(l_{\ast_1}(y)a-r_{\ast_1}(y)a)x
=y\ast_1(r_{\prec_2}(a)x)+r_{\ast_2}(l_{\prec_1}(x)a)y-r_{\prec_2}(a)(y\ast_1 x),
\\&\label{ppmp eq1.8}
x\prec_1(l_{\ast_2}(a)y-r_{\ast_2}(a)y)+r_{\prec_2}(r_{\ast_1}(y)a-l_{\ast_1}(y)a)x
=l_{\ast_2}(a)(x\prec_1y)-(l_{\ast_2}(a)x)\prec_1 y-l_{\prec_2}(r_{\ast_1}(x)a)y,
\\&\label{ppmp eq1.9}l_{\prec_2}(a)(x\ast_1y-y\ast_1x)=x\ast_1(l_{\prec_2}(a)y)+r_{\ast_2}(r_{\prec_1}(y)a)x-
(r_{\ast_2}(a)x)\prec_1y-l_{\prec_2}(l_{\ast_1}(x)a)y,
\\&\label{ppmp eq1.10}r_{\ast_2}(a)(x\cdot_1y)=(r_{\ast_2}(a)x)\succ_1 y+l_{\prec_2}(l_{\ast_1}(x)a)y+x\succ_1(r_{\ast_2}(a)y)
+r_{\succ_2}(l_{\ast_1}(y)a)x,
\\&\label{ppmp eq1.11}(r_{\cdot_2}(a)x)\ast_1y+l_{\ast_2}(l_{\cdot_1}(x)a)y
=r_{\prec_2}(a)(x\ast_1y)+x\succ_1(l_{\ast_2}(a)y)+r_{\succ_2}(r_{\ast_1}(y)a)x,
\\&\label{ppmp eq1.12}(l_{\cdot_2}(a)x)\ast_1y+l_{\ast_2}(r_{\cdot_1}(x)a)y
=(l_{\ast_2}(a)y)\prec_1 x+l_{\prec_2}(r_{\ast_1}(y)a)x+l_{\succ_2}(a)(x\ast_1y),\\&
\label{ppmp eq1.13}r_{\succ_1}(x)[a, b]_2=a\ast_{2}
		r_{\succ_1}(x)b+r_{\ast_1}(l_{\succ_2}(b)x)a-b\succ_{2}(r_{\ast_1}(x)a)-r_{\succ_1}(l_{\ast_2}(a)x)b,
\\&\label{ppmp eq1.14}
(r_{\ast_1}(x)a-l_{\ast_1}(x)a)\succ_2 b+l_{\succ_1}(l_{\ast_2}(a)x-r_{\ast_2}(a)x)b=a\ast_2(l_{\succ_1}(x)b)+
r_{\ast_1}(r_{\succ_2}(b)x)a-l_{\succ_1}(x)(a\ast_2 b),
\\&\label{ppmp eq1.15}(l_{\ast_1}(x)a-r_{\ast_1}(x)a)\succ_2 b-l_{\succ_1}(r_{\ast_2}(a)x-l_{\ast_2}(a)x)b=
l_{\ast_1}(x)(a\succ_2 b)-a\succ_2(l_{\ast_1}(x)b)
-r_{\succ_1}(r_{\ast_2}(b)x)a,
\\&\label{ppmp eq1.16}a\prec_2(r_{\ast_1}(x)b-l_{\ast_1}(x)b)+r_{\prec_1}(l_{\ast_2}(b)x-r_{\ast_2}(b)x)a
=b\ast_2(r_{\prec_1}(x)a)+r_{\ast_1}(l_{\prec_2}(a)x)b-r_{\prec_1}(x)(b\ast_1 a),
\\&\label{ppmp eq1.17}
a\prec_2(l_{\ast_1}(x)b-r_{\ast_1}(x)b)+r_{\prec_1}(r_{\ast_2}(b)x-l_{\ast_2}(b)x)a
=l_{\ast_1}(x)(a\prec_2b)-(l_{\ast_1}(x)a)\prec_2 b-l_{\prec_1}(r_{\ast_2}(a)x)b,
\\&\label{ppmp eq1.18}l_{\prec_1}(x)(a\ast_2b-b\ast_2a)=a\ast_2(l_{\prec_1}(x)b)+r_{\ast_1}(r_{\prec_2}(b)x)a-
(r_{\ast_1}(x)a)\prec_2b-l_{\prec_1}(l_{\ast_2}(a)x)b,
\\&\label{ppmp eq1.19}r_{\ast_1}(x)(a\cdot_2b)=(r_{\ast_1}(x)a)\succ_2 b+l_{\prec_1}(l_{\ast_2}(a)x)b+a\succ_2(r_{\ast_1}(x)b)
+r_{\succ_1}(l_{\ast_2}(b)x)a,
\\&\label{ppmp eq1.20}(r_{\cdot_1}(x)a)\ast_2b+l_{\ast_1}(l_{\cdot_2}(a)x)b
=r_{\prec_1}(x)(a\ast_2b)+a\succ_2(l_{\ast_1}(x)b)+r_{\succ_1}(r_{\ast_2}(b)x)a,
\\&\label{ppmp eq1.21}(l_{\cdot_1}(x)a)\ast_2b+l_{\ast_1}(r_{\cdot_2}(a)x)b
=(l_{\ast_1}(x)b)\prec_2 a+l_{\prec_1}(r_{\ast_2}(b)x)a+l_{\succ_1}(x)(a\ast_2b),\end{align}
\end{small}
\end{enumerate}
where $\cdot_1=\succ_1+\prec_1,~\cdot_2=\succ_2+\prec_2,~ [x,y]=x\ast_1 y-y\ast_1 x $ 
and $[a,b]=a\ast_2 b-b\ast_2 a $.
In this case, we
denote this noncommutative pre-Poisson algebra by $A_1\bowtie A_2$, and 
$(A_1, A_2,
l_{\succ_1},r_{\succ_1},l_{\prec_1},r_{\prec_1},l_{\ast_1},r_{\ast_1},$ \ \ \ $
l_{\succ_2},r_{\succ_2},l_{\prec_2},r_{\prec_2},l_{\ast_2},r_{\ast_2})$
satisfying the above conditions is called a {\bf matched pair of
noncommutative pre-Poisson algebras}. \end{pro}

\begin{proof} To prove that $(A_1\oplus A_2,\succ,\prec,\ast)$ is a noncommutative pre-Poisson algebra,
		we only need to check that Eqs.~(\ref{Np1})-(\ref{Np3}) hold for 
all $x+a,y+b,z+c\in A_1\oplus A_2$ with $x,y,z\in A_1,a,b,c\in A_2$. 
Indeed, we may verify the following cases through direct computations:
\begin{enumerate}
\item Eq.~(\ref{Np1}) holds for $(x,y,a)$ if and only if Eqs.~(\ref {r1}) and (\ref {ppmp eq1.4}) hold.
\item Eq.~(\ref{Np1}) holds for $(x,a,y)$ if and only if Eqs.~(\ref {r2}) and (\ref {ppmp eq1.5}) hold.
\item Eq.~(\ref{Np1}) holds for $(a,x,y)$ if and only if Eqs.~(\ref {r7}) and (\ref {ppmp eq1.6}) hold.
\item Eq.~(\ref{Np1}) holds for $(a,b,x)$ if and only if Eqs.~(\ref {r1}) and (\ref {ppmp eq1.13}) hold.
\item Eq.~(\ref{Np1}) holds for $(a,x,b)$ if and only if Eqs.~(\ref {r2}) and (\ref {ppmp eq1.14}) hold.
\item Eq.~(\ref{Np1}) holds for $(x,a,b)$ if and only if Eqs.~(\ref {r7}) and (\ref {ppmp eq1.15}) hold.
\item Eq.~(\ref{Np2}) holds for $(x,y,a)$ if and only if Eqs.~(\ref {r3}) and (\ref {ppmp eq1.7}) hold.
\item Eq.~(\ref{Np2}) holds for $(x,a,y)$ if and only if Eqs.~(\ref {r8}) and (\ref {ppmp eq1.8}) hold.
\item Eq.~(\ref{Np2}) holds for $(a,x,y)$ if and only if Eqs.~(\ref {r4}) and (\ref {ppmp eq1.9}) hold.
\item Eq.~(\ref{Np2}) holds for $(a,b,x)$ if and only if Eqs.~(\ref {r3}) and (\ref {ppmp eq1.16}) hold.
\item Eq.~(\ref{Np2}) holds for $(a,x,b)$ if and only if Eqs.~(\ref {r8}) and (\ref {ppmp eq1.17}) hold.	
\item Eq.~(\ref{Np2}) holds for $(x,a,b)$ if and only if Eqs.~(\ref {r4}) and (\ref {ppmp eq1.18}) hold.
\item Eq.~(\ref{Np3}) holds for $(x,y,a)$ if and only if Eqs.~(\ref {r9}) and (\ref {ppmp eq1.10}) hold.
\item Eq.~(\ref{Np3}) holds for $(x,a,y)$ if and only if Eqs.~(\ref {r5}) and (\ref {ppmp eq1.11}) hold.
\item Eq.~(\ref{Np3}) holds for $(a,x,y)$ if and only if Eqs.~(\ref {r6}) and (\ref {ppmp eq1.12}) hold.
\item Eq.~(\ref{Np3}) holds for $(a,b,x)$ if and only if Eqs.~(\ref {r9}) and (\ref {ppmp eq1.19}) hold.
\item Eq.~(\ref{Np3}) holds for $(a,x,b)$ if and only if Eqs.~(\ref {r5}) and (\ref {ppmp eq1.20}) hold.	
\item Eq.~(\ref{Np3}) holds for $(x,a,b)$ if and only if Eqs.~(\ref {r6}) and (\ref {ppmp eq1.21}) hold.
\end{enumerate}
\end{proof}

\begin{rmk} \label{Ms} 
\begin{enumerate}
\item If the linear maps
 $l_{\succ_2},r_{\succ_2},l_{\prec_2},r_{\prec_2},l_{\ast_2},r_{\ast_2}: A_{2}\rightarrow \text{End}(A_1)$ are trivial, then the matched pair 
reduces to an A-noncommutative pre-Poisson algebra. 
For the completeness, we provide the details. Let $(A_{1},\succ_{1},\prec_{1},\ast_1)$ and
$(A_{2},\succ_{2},\prec_{2},\ast_2)$ be two noncommutative pre-Poisson algebras. Assume that there are linear maps
\begin{equation*}l_{\succ_1},r_{\succ_1},l_{\prec_1},r_{\prec_1},l_{\ast_1},r_{\ast_1}:A_1\longrightarrow \hbox{End}(A_2)\end{equation*}
 such that
$(A_2,l_{\succ_1},r_{\succ_1},l_{\prec_1},r_{\prec_1},l_{\ast_1},r_{\ast_1})$ 
is a representation of $(A_1,\succ_1,\prec_1,\ast_1)$,
and the following compatible conditions hold for all $x,y\in A_1,a,b\in A_2$:
\begin{small}
	\begin{align}
	&\label{ppmp eq1.22}r_{\succ_1}(x)(a\ast_{2} b-b\ast_{2} a)=a\ast_{2}r_{\succ_1}(x)b-b\succ_{2}(r_{\ast_1}(x)a),
\\&\label{ppmp eq1.23}(r_{\ast_1}(x)a-l_{\ast_1}(x)a)\succ_2 b=a\ast_2(l_{\succ_1}(x)b)-l_{\succ_1}(x)(a\ast_2 b),
\\&\label{ppmp eq1.24}(l_{\ast_1}(x)a-r_{\ast_1}(x)a)\succ_2 b=l_{\ast_1}(x)(a\succ_2 b)-a\succ_2(l_{\ast_1}(x)b),
\\&\label{ppmp eq1.25}a\prec_2(r_{\ast_1}(x)b-l_{\ast_1}(x)b)=b\ast_2(r_{\prec_1}(x)a)-r_{\prec_1}(x)(b\ast_1 a),
\\&\label{ppmp eq1.26}a\prec_2(l_{\ast_1}(x)b-r_{\ast_1}(x)b)=l_{\ast_1}(x)(a\prec_2b)-(l_{\ast_1}(x)a)\prec_2 b,
\\&\label{ppmp eq1.27}l_{\prec_1}(x)(a\ast_2b-b\ast_2a)=a\ast_2(l_{\prec_1}(x)b)-(r_{\ast_1}(x)a)\prec_2b,
\\&\label{ppmp eq1.28}r_{\ast_1}(x)(a\cdot_2b)=(r_{\ast_1}(x)a)\succ_2 b+a\succ_2(r_{\ast_1}(x)b),
\\&\label{ppmp eq1.29}(r_{\cdot_1}(x)a)\ast_2b=r_{\prec_1}(x)(a\ast_2b)+a\succ_2(l_{\ast_1}(x)b),
\\&\label{ppmp eq1.30}(l_{\cdot_1}(x)a)\ast_2b=(l_{\ast_1}(x)b)\prec_2 a+l_{\succ_1}(x)(a\ast_2b),
\\&\label{Dl1}r_{\ast_1}(x)(a\ast_{2} b-b\ast_{2} a)=a\ast_2(r_{\ast_1}(x)b)-b\ast_2(r_{\ast_1}(x)a),
\\&\label{Dl2}l_{\ast_1}(x)(a\ast_2 b)=(l_{\ast_1}(x)-r_{\ast_1}(x)a)\ast_2 b+a\ast_2(l_{\ast_1}(x)b),
\\&\label{Dl3}(l_{\prec_1}(x)a)\prec_2 b=l_{\prec_1}(x)(a\cdot_2 b), \ \ \  (l_{\succ_1}(x)a)\prec_2 b=l_{\succ_1}(x)(a\prec_2 b),\ \ \
\\&\label{Dl4} l_{\succ_1}(x)(a\succ_2 b)=(l_{\cdot_1}(x)a)\succ_2 b, \ \ \ (r_{\prec_1}(x)a)\prec_2 b=a\prec_2(l_{\cdot_1}(x) b),
\\&\label{Dl5} (r_{\succ_1}(x)a)\prec_2 b=a\succ_2(l_{\prec_1}(x) b),\ \ \  a\succ_2l_{\succ}(x)b=r_{\cdot}(x)a\succ_2 b,
\\&\label{Dl6} r_{\prec_1}(x)(a\prec_2 b)=a\prec_2(r_{\cdot_1}(x) b), \ \ \ r_{\prec_1}(x)(a\succ_2 b)=a\succ_2(r_{\prec_1}(x)b),
\\&\label{Dl7} a\succ_2(r_{\succ_1}(x)b)=r_{\succ_1}(x)(a\cdot_2 b).
\end{align}\end{small}
Then $(A_{2},l_{\succ_1},r_{\succ_1},l_{\prec_1},r_{\prec_1},l_{\ast_1},r_{\ast_1})$ is called
  an {\bf A-noncommutative pre-Poisson algebra}. If only Eqs.~(\ref{Dl1})-(\ref{Dl2}) are satisfied, $(A_{2},l_{\ast_1},r_{\ast_1})$ is called
  an {\bf A-pre-Lie algebra}. If only Eqs.~(\ref{Dl3})-(\ref{Dl7}) are satisfied, $(A_{2},l_{\succ_1},r_{\succ_1},l_{\prec_1},r_{\prec_1})$ is called
   an {\bf  A-dendriform algebra}.  A-dendriform algebra is named an action of $A_1$ on $A_{2}$ in \cite{Wa}.
  \item If the linear maps
 $l_{\succ_2},r_{\succ_2},l_{\prec_2},r_{\prec_2},l_{\ast_2},r_{\ast_2}: A_{2}\rightarrow \text{End}(A_1)$
  and $(\succ_{2},\prec_{2},\ast_2)$ are trivial, 
 then the matched pair 
reduces to the semidirect product. Denote it simply by $A_{1}\ltimes A_2$. In particular, if in addition, $A_1$ is coherent
and $(l_{\succ_1},r_{\succ_1},l_{\prec_1},r_{\prec_1},l_{\ast_1},r_{\ast_1})$ satisfying Eq.~(\ref{r10})-(\ref{r12}), 
then $A_1\ltimes A_2$ and $A_1\ltimes A_2^{*}$ both are coherent noncommutative pre-Poisson algebras.
  \end{enumerate}
\end{rmk}
\subsection{Phase spaces of noncommutative Poisson algebras}
We now introduce the notions of symplectic noncommutative Poisson algebras and the phase spaces of noncommutative Poisson algebras.

\begin{defi} A {\bf symplectic noncommutative Poisson algebra} is a noncommutative Poisson algebra $(A,\cdot, [ \ , \ ])$
together with a non-degenerate skew-symmetric bilinear form $\omega$ satisfying the following conditions for all $x,y,z\in A$,
\begin{align}&
\label{Bs1}\omega([x,y],z)+\omega([y,z],x)+\omega([z,x],y)=0,\\&
\label{Bs2}\omega (x\cdot y,z)+\omega (y\cdot z,x)+\omega (z\cdot x,y)=0,
\end{align}
that is, $(A, [ \ , \ ],\omega)$
is a symplectic Lie algebra \cite{B1} and $(A,\cdot,\omega)$ is an associative algebra with a Connes cocycle \cite{B2}.
\end{defi}

\begin{defi} A {\bf quadratic noncommutative pre-Poisson algebra} is a noncommutative pre-Poisson algebra $(A,\succ,\prec, \ast)$
endowed with a nondegenerate 
skew-symmetric bilinear form $\omega:A\otimes A\longrightarrow k$ satisfying the following invariance conditions for all
 $x,y,z\in A $
  \begin{align}&\label{Ib1}
    \omega(x\ast y,z)=-\omega(y,[x,z]),\\&
    \label{Ib2} \omega(x\succ y,z)=\omega(y,z\cdot x) ,\ \ \ 
    \omega(x\prec y,z)=\omega(x,y\cdot z),
  \end{align}
where $x\ast z-z\ast x=[x,z]$ and $x\cdot y=x\succ y+x\prec y$.\end{defi}

\begin{ex} Let $(A,\succ,\prec, \ast)$ be the 2-dimensional noncommutative pre-Poisson algebra with a basis 
$\{ e_1,e_2\}$ given in Example \ref{Ae}.
Define a nondegenerate 
skew-symmetric bilinear form $\omega$ on $A$ as follows:
$\omega(e_1,e_2)=-\omega(e_2,e_1)=1 $ and $\omega(e_2,e_2)=\omega(e_1,e_1)=0 $.
Then $(A,\succ,\prec, \ast,\omega)$ is a quadratic noncommutative pre-Poisson algebra.
\end{ex}

\begin{pro}\label{Ps1} Let  $(A, \cdot, [ \ , \ ],\omega)$ be a symplectic noncommutative Poisson algebra.
Then there is a compatible coherent
noncommutative pre-Poisson algebra $(A,\succ,\prec, \ast)$ given by equations~\eqref{Ib1}-\eqref{Ib2}.
Conversely, assume that $(A, \succ,\prec,\ast,\omega)$ is a quadratic noncommutative pre-Poisson algebra
and $(A, \cdot, [ \ , \ ])$ is
its sub-adjacent noncommutative Poisson algebra. Then $(A, \cdot, [ \ , \ ],\omega)$ is 
a symplectic noncommutative Poisson algebra.\end{pro}
 
 \begin{proof} By Theorem 4.1.1 \cite{B2} and Theorem 2.2 \cite{B1}, we get that $(A,\succ,\prec)$ is a dendriform algebra
and $(A,\ast)$ is a pre-Lie algebra. Using Eqs.~\eqref{Np0} and \eqref{Ib1}-\eqref{Ib2}, we have for all $x,y,z,w\in A$, 
 \begin{align*}&
 \omega((x\ast y-y\ast x)\succ z,w)-\omega(x\ast(y\succ z)-y\succ(x\ast z),w)\\
   =&\omega(z,w\cdot[x, y])+\omega(z,[x, w]\cdot y)
   -\omega(z,[x,w\cdot y])
   \\=&0.
  \end{align*}
Thus, Eq.~\eqref{Np1} holds. Analogously, Eqs.~\eqref{Np2}-\eqref{Np3} hold. 
By Eqs.~\eqref{Ib1}-\eqref{Ib2} and \eqref{Np0}, we get 
\begin{align*}& \omega((x\cdot y)\ast z-x\ast(y\succ z)-y\ast(z\prec x),w)
\\=&-\omega(z,[x\cdot y,w])+\omega(y\succ z,[x, w])+\omega(z\prec x,[y,w])
\\=&-\omega(z,[x\cdot y,w])+\omega( z,[x, w]\cdot y)+\omega(z, x\cdot[y,w])
\\=&0,
\end{align*}
which yields that Eq.~\eqref{Np4} holds, that is, $(A,\succ,\prec, \ast)$ is coherent.
The converse part is easy to check.
The proof is completed.
\end{proof}

\begin{defi}
Let $(A,\cdot_A, [ \ , \ ]_A)$ be a noncommutative Poisson algebra and $A^{*}$ its dual space. If there is a
noncommutative Poisson algebra structure $(\cdot, [ \ , \ ])$ on the direct sum vector space $A\oplus A^{*}$ such that
$(A\oplus A^{*},\cdot, [ \ , \ ],\omega)$ is a symplectic noncommutative Poisson algebra, where $\omega$ is given by the following equation:
 \begin{equation}\label{C2}\omega(x+a, y+b) =\langle a,y\rangle-\langle x,b\rangle, \forall~x, y\in A,~a, b\in A^{*},\end{equation}
and both $(A,\cdot_A,[ \ , \ ]_A)$, $(A^{*},\cdot|_{A^{*}},[ \ , \ ]|_{A^{*}})$ are noncommutative Poisson subalgebras 
of $(A\oplus A^{*},\cdot, [ \ , \ ])$, then the symplectic noncommutative Poisson
algebra $(A\oplus A^{*},\cdot, [ \ , \ ],\omega)$ is called a {\bf phase space} of the noncommutative Poisson algebra $(A,\cdot_A,[ \ , \ ]_A)$.
\end{defi}

\begin{pro}\label{Ps2} A noncommutative Poisson algebra has a phase space if and only if there is a compatible
 coherent noncommutative pre-Poisson algebra.
\end{pro}

\begin{proof} Assume that $(A\oplus A^{*},\cdot, [ \ , \ ],\omega)$ is a phase space of
 the noncommutative Poisson algebra $(A,\cdot_A,[ \ , \ ]_A)$.
 In light of Proposition \ref{Ps1}, there is a compatible coherent noncommutative
  pre-Poisson algebra structure $(\succ,\prec,\ast)$ on $A\oplus A^{*}$ defined by Eqs.~\eqref{Ib1}-\eqref{Ib2}.
Due to $(A,\cdot_A,[ \ , \ ]_A)$ being a noncommutative Poisson subalgebra of $A\oplus A^{*}$, we obtain for all $x,y,z\in A$,
 \begin{align*}&
    \omega(x\ast y,z)=-\omega(y,[x,z])=0,\\&
  \omega(x\succ y,z)=\omega(y,z\cdot x)=0 ,\ \ \ 
    \omega(x\prec y,z)=\omega(x,y\cdot z)=0,
  \end{align*}
which means that $x\ast y,x\succ y,x\prec y\in A$, which indicates that
 $(A,\succ|_{A},\prec|_{A},\ast|_{A})$ and $(A,\succ_A,\prec_A,\ast_A)$ are consistent. Thus, the
 sub-adjacent noncommutative Poisson algebra of $(A,\succ|_{A},\prec|_{A},\ast|_{A})$ is exactly $(A,\cdot_A,[ \ , \ ]_A)$.
 
Conversely, assume that $(A,\succ_A,\prec_A, \ast_A)$ is a coherent noncommutative pre-Poisson algebra. 
By Example \ref{Cr}, we know that $(A^{*},R_{\prec}^{*}, L_{\succ}^{*},-L_{\ast}^{*})$ is a representation of $(A,\cdot_A, [ \ , \ ]_A)$.
 It follows that $(A\ltimes A^{*},\cdot,[ \ , \ ])$ is a noncommutative Poisson algebra,
 where 
 \begin{align*}&
 (x+a)\cdot (y+b)=x\succ_A y+R_{\prec}^{*}(x)b+L_{\succ}^{*}(y)a,\\&
 [x+a, y+b]=[x,y]_A -L_{\ast}^{*}(x)b+L_{\ast}^{*}(y)a, 
  \end{align*} 
  for all $x,y\in A,a,b\in A^{*}$.
 It is easy to check that $(A\ltimes A^{*},\cdot,[ \ , \ ],\omega)$ is a
   a phase space of the noncommutative Poisson algebra $(A,\cdot_A,[ \ , \ ]_A)$, where $\omega$ is defined by Eq.~\eqref{C2}.
\end{proof}

\begin{cor} \label{Ps3} Let $(A\oplus A^{*},\cdot, [ \ , \ ],\omega)$ be a phase space of
 the noncommutative Poisson algebra $(A,\cdot_A,[ \ , \ ]_A)$
and $(A\oplus A^{*},\succ,\prec,\ast)$ the compatible noncommutative pre-Poisson algebra. Then both 
$(A,\succ|_{A},\prec|_{A},\ast|_{A})$ and $(A^{*},\succ|_{A^{*}},\prec|_{A^{*}},\ast|_{A^{*}})$
are subalgebras of the noncommutative pre-Poisson algebra $(A\oplus A^{*},\succ,\prec,\ast)$.
\end{cor}

\begin{proof} This follows directly from the proof of Proposition \ref{Ps2}. \end{proof}

\begin{defi}\label{defi:Manin}
  A {\bf Manin triple of noncommutative pre-Poisson algebras} is a quadruple $(A,A_1,A_2,\omega)$, 
  where $(A,\omega)$ is a quadratic noncommutative pre-Poisson algebra, 
  $(A_1,\succ_1,\prec_1,\ast_1)$ and $(A_2,\succ_2,\prec_2,\ast_2)$ are noncommutative pre-Poisson subalgebras of $A$,
  such that
  \begin{enumerate}
\item $A=A_1\oplus A_2$ as vector spaces;
\item $A_1$ and $A_2$ are isotropic with respect to $\omega$.
\end{enumerate}
 \end{defi}
 
\begin{pro} \label{Ps4}
Let $(A,A_1,A_2,\omega)$ be a Manin triple of noncommutative pre-Poisson algebras. 
Then the noncommutative pre-Poisson algebras $A$, $A_1$ and $A_2$ are coherent.
\end{pro}
\begin{proof} By Eqs.~\eqref{Ib1}-\eqref{Ib2}  and \eqref{Np1}-\eqref{Np3},
 we have for all $x,y,z\in A_1$ and $a\in A_2$, 
\begin{align*}&\omega((x\cdot y)\ast z,a)-\omega(x\ast(y\succ z),a)-\omega(y\ast(z\prec x),a)
\\=&-\omega(z,(x\cdot y)\ast a-a\ast(x\cdot y))
+\omega(y\succ z,x\ast a-a\ast x)+\omega( z\prec x,y\ast a-a\ast y)
\\=&-\omega(z,(x\cdot y)\ast a-a\ast(x\cdot y))
+\omega( z,(x\ast a-a\ast x)\cdot y)
+\omega( z,x\cdot (y\ast a-a\ast y))
\\=&0,\end{align*}
which indicates that Eq.~\eqref{Np4} holds. Thus, $A_1$ is coherent. Analogously,
$A$ and $A_2$ are coherent.
\end{proof}

\begin{thm} \label{Mp1}
Suppose that $(A,\succ_{A},\prec_{A},\ast_{A})$ and $(A^{*},\succ_{A^{*}},\prec_{A^{*}},\ast_{A^{*}})$ are two 
 noncommutative pre-Poisson algebras and their sub-adjacent
noncommutative Poisson algebras are $(A,\cdot_A, [ \ , \ ]_A)$ and
 $(A^{*},\cdot_{A^{*}}, [ \ , \ ]_{A^{*}})$ respectively. Then the following conditions are equivalent:
\begin{enumerate}
\item \label{Pm1} $(A\oplus A^{*},\cdot,[ \ , \ ],\omega)$ is a phase space of the noncommutative Poisson algebra $(A,\cdot_A, [ \ , \ ]_A)$.
\item \label{Pm2} $(A\oplus A^{*},A, A^{*},\omega)$ is a Manin triple of noncommutative pre-Poisson algebras with the
bilinear form given by Eq.~\eqref{C2}
and the isotropic subalgebras are $A$ and $ A^{*}$.
\item \label{Pm3}
$(A, A^{*}, R_{\cdot_1}^{*},-L_{\prec_1}^{*},-R_{\succ_1}^{*},L_{\cdot_1}^{*},-\mathrm{ad_1}^{*},R^{*}_{\ast_1},
 R_{\cdot_2}^{*},-L_{\prec_2}^{*},-R_{\succ_2}^{*},L_{\cdot_2}^{*},-\mathrm{ad_2}^{*},R^{*}_{\ast_2})$ 
 is a matched pair of coherent noncommutative pre-Poisson algebras.
 \item \label{Pm4}
$(A, A^{*},R_{\prec_1}^{*}, L_{\succ_1}^{*},-L_{\ast_1}^{*},R_{\prec_2}^{*},L_{\succ_2}^{*},-L_{\ast_2}^{*}) $ is a matched pair of coherent 
noncommutative Poisson algebras.
\end{enumerate}
\end{thm}

\begin{proof} \eqref{Pm1} $\Longrightarrow$ \eqref{Pm2}
Assume that $(A\oplus A^{*},\cdot,[ \ , \ ],\omega)$ is a phase 
space of the noncommutative Poisson algebra $(A,\cdot_A, [ \ , \ ]_A)$.
According to Proposition \ref{Ps1}, there is a coherent noncommutative pre-Poisson algebra structure $(\succ,\prec,\ast)$
on $A\oplus A^{*}$ defined by Eqs.~\eqref{Ib1}-\eqref{Ib2} such that 
 $(A\oplus A^{*},\succ,\prec,\ast)$ is quadratic noncommutative pre-Poisson algebra.
 In view of Corollary \ref{Ps3}, $(A,\succ|_{A},\prec|_{A},\ast|_{A})$ and $(A^{*},\succ|_{A^{*}},\prec|_{A^{*}},\ast|_{A^{*}})$
are subalgebras of $(A\oplus A^{*},\succ,\prec,\ast)$.
It is obvious that $A$ and $ A^{*}$ are isotropic. Thus, $(A\oplus A^{*},A, A^{*},\omega)$ is
 a Manin triple of noncommutative pre-Poisson algebras.
 Combining Proposition 4.2 \cite{B1}, Theorem 4.1.5 \cite{B2} and Proposition \ref{Mp}, it follows directly that
\eqref{Pm2} $\Longleftrightarrow$ \eqref{Pm3} $\Longleftrightarrow$ \eqref{Pm4} $\Longrightarrow$ \eqref{Pm1}. 
\end{proof}

\section{Noncommutative pre-Poisson bialgebras and the noncommutative pre-Poisson Yang-Baxter equation}

\subsection{Dendriform bialgebras and pre-Lie bialgebras}
We start by reviewing the bialgebra theories for dendriform algebras (see \cite{B2}) and pre-Lie bialgebras (see \cite{B1}).
\begin{defi} 
	\begin{enumerate}
		\item A {\bf dendriform coalgebra} is a triple $(A,\Delta_{\succ},\Delta_{\prec})$, where
$A$ is a vector space and $\Delta_{\succ},\Delta_{\prec}:A\longrightarrow A\otimes A$ are linear maps such that
the following conditions hold:
\begin{align*}
( \Delta_{\succ}\otimes I)\Delta_{\prec}=(I\otimes\Delta_{\prec})\Delta_{\succ},\ \ \ 
(I\otimes \Delta_{\succ})\Delta_{\succ}=(\Delta\otimes I)\Delta_{\succ}, \ \ \
( \Delta_{\prec}\otimes I)\Delta_{\prec}=(I\otimes\Delta)\Delta_{\prec},\end{align*}
where $\Delta=\Delta_{\succ}+\Delta_{\prec}.$
\item A {\bf dendriform bialgebra} is a quintuple $(A,\succ,\prec,\Delta_{\succ},\Delta_{\prec})$,
 where $(A,\succ,\prec)$ is a dendriform algebra and $(A,\Delta_{\succ},\Delta_{\prec}
		)$ is a dendriform coalgebra such that for all $x,y\in A$,
	 the following compatible conditions hold:
\begin{align*}&\Delta_{\prec}(x\cdot_{A} y)=((I\otimes L_{\succ_{A}} (x))\Delta_{\prec_{A}}(y)+R_{\cdot_{A}}(y)\otimes I)\Delta_{\prec}(x),\\&
\Delta_{\succ}(x\cdot_{A} y)=(I\otimes L_{\cdot_{A}}(x))\Delta_{\succ}(y)+(R_{\prec_{A}} (y)\otimes I)\Delta_{\succ}(x),\\&
\Delta(x\prec y)=(R_{\prec} (y)\otimes I)\Delta(x)+(I\otimes L_{\prec} (x))\Delta_{\succ}(y),\\&
\Delta(x\succ y)=(I\otimes L_{\succ} (x))\Delta(y)+(R_{\succ} (y)\otimes I)\Delta_{\prec}(x),\\&
(L_{\cdot_{A}} (x)\otimes I-(I\otimes R_{\prec_{A}} (x))\Delta_{\prec}(y)
+\tau(L_{\succ_{A}} (y)\otimes I-I\otimes  R_{\cdot_{A}} (y))\Delta_{\succ}(x)=0,\\&
(L_{\succ} (x)\otimes I)\Delta(y)-(R_{\succ} (y)\otimes I)\tau\Delta_{\succ}(x)
+(I\otimes L_{\prec} (y))\tau\Delta_{\prec}(x)-(I\otimes R_{\prec} (x))\Delta(y)=0,\end{align*}
where $\cdot=\succ+\prec,~\Delta=\Delta_{\succ}+\Delta_{\prec}$ and $R_{\cdot}=R_{\prec}+R_{\succ},~L_{\cdot}=L_{\prec}+L_{\succ}$.
\end{enumerate}
\end{defi}

\begin{thm} \label{Cg1} Let $(A,\succ,\prec)$ be a dendriform algebra and 
$r=\sum_{i}a_i\otimes b_i\in A\otimes A$. Assume that
$\Delta_{\succ,r},\Delta_{\prec,r}$ defined by 
\begin{align}&\label{CB1}
\Delta_{\succ,r}(x)=(I\otimes L_{\cdot}(x)-R_{\prec}(x)\otimes I)\tau(r),
\\&\label{CB2}\Delta_{\prec,r}(x)=(R_{\cdot}(x)\otimes I-I\otimes L_{\succ}(x))r,
\end{align}
Then $(A,\succ,\prec,\Delta_{\succ,r},\Delta_{\prec,r})$ is a dendriform bialgebra if and only if the following equations hold:
\begin{align}\label{Pa1}&[Q(x\succ y)-I\otimes L_{\succ}(x)Q(y)](r-\tau(r))=0,\\&
\label{Pa2}\tau(Q(x))Q(y)(r-\tau(r))=0,\\&
\label{Pa3}(R_{\cdot}(x)\otimes I \otimes I-I \otimes I\otimes L_{\succ}(x))(r_{12}\cdot r_{13}-r_{23}\succ r_{12}-r_{13}\prec r_{32})
\\&+\sum_{i}(a_i\cdot x)\otimes Q(a_i)(r-\tau(r))-a_i\otimes[Q(x\succ b_i)(r-\tau(r))]=0,\nonumber\\&
\label{Pa4}(R_{\prec}(x)\otimes I \otimes I-I \otimes I\otimes L_{\succ}(x))(r_{31}\succ r_{23}+r_{21}\prec r_{13}-r_{23}\cdot r_{21})=0,
\\&
\label{Pa5}(R_{\prec}(x)\otimes I \otimes I-I \otimes I\otimes L_{\cdot}(x))(r_{32}\prec r_{21}+r_{12}\succ r_{31}-r_{31}\cdot r_{32})
\\&+[Q(b_i)(r-\tau(r))\otimes x\cdot a_i-Q( b_i\prec x)(r-\tau(r))\otimes a_i]=0,\nonumber\end{align}
where $Q(x)=I\otimes L_{\succ}(x)-R_{\prec}(x)\otimes I$.
\end{thm}

\begin{pro} \label{Cg2}
Let $(A,\succ,\prec)$ be a dendriform algebra and $r=\sum\limits_{i}a_{i}\otimes b_{i}\in A\otimes A$ symmetric.
Define linear maps $\Delta_{\succ,r},\Delta_{\prec,r}:A\longrightarrow A\otimes A$ 
by Eqs.~(\ref{CB1})-(\ref{CB2}).
Then $(A, \succ,\prec,\ast,\Delta_{\succ,r},\Delta_{\prec,r})$ 
 is a dendriform bialgebra if and only if
  the following equations hold:
\begin{align}\label{Dy1}&
(R_{\cdot}(x)\otimes I \otimes I-I \otimes I\otimes L_{\succ}(x))(r_{12}\cdot r_{13}-r_{23}\succ r_{12}-r_{13}\prec r_{23})=0,\\&
\label{Dy2}(R_{\prec}(x)\otimes I \otimes I-I \otimes I\otimes L_{\succ}(x))(r_{13}\succ r_{23}+r_{12}\prec r_{13}-r_{23}\cdot r_{12})=0,
\\&
\label{Dy3}(R_{\prec}(x)\otimes I \otimes I-I \otimes I\otimes L_{\cdot}(x))(r_{23}\prec r_{12}+r_{12}\succ r_{13}-r_{13}\cdot r_{23})
=0.\end{align}
\end{pro}

Let $(A,\succ,\prec)$ be a dendriform algebra and $r\in A\otimes A$. The equation 
\begin{equation} \label{YE6} D(r)=r_{12}\cdot r_{13}-r_{23}\succ r_{12}-r_{13}\prec r_{23}=0\end{equation}
is called the D-equation in $(A,\succ,\prec)$.

\begin{thm} \label{Zb1}
Let $(A,\succ,\prec)$ be a dendriform algebra and $r\in A\otimes A$.
If $r$ is a symmetric solution of the D-equation in $(A,\succ,\prec)$, then
 $(A,\succ,\prec,\Delta_{\succ,r},\Delta_{\prec,r})$ is a dendriform bialgebra.
\end{thm}

\begin{defi}
	\begin{enumerate}
		\item A {\bf pre-Lie coalgebra} is a vector space $A$ together
		with a linear map $\delta: A \longrightarrow A\otimes A$ such that
\begin{small}
		\begin{align*}&
			(I\otimes\delta)\delta-(\tau\otimes I)(I\otimes\delta)\delta=(\delta\otimes I)\delta-(\tau\otimes I)(\delta\otimes I)\delta.\end{align*}
\end{small}
	\item A {\bf pre-Lie bialgebra} is a triple $(A,\ast,
	\delta)$, where $(A,\ast)$ is a pre-Lie algebra and $(A,
	\delta)$ is a pre-Lie coalgebra such that 
	 the following compatible conditions hold for all $x,y\in A$,
	\begin{small}
\begin{align*}&
		\delta(x\ast y-y\ast x)=(I\otimes \mathrm{ad}(x)+L_\ast(x) \otimes I)\delta(y)-
(I\otimes \mathrm{ad}(y)+L_\ast(y) \otimes I)\delta(x),\\
		&(\delta-\tau \delta)(x\ast y)=L_\ast(x) \otimes I+I \otimes L_\ast(x))(\delta-\tau \delta)(y)
+(I\otimes L_\ast(y))\delta(x)-(L_\ast(y)\otimes I)\tau\delta(x).
	\end{align*}
\end{small}
	\end{enumerate}	
\end{defi}

\begin{thm}\label{Cg3}
Let $(A,\ast)$ be a pre-Lie algebra and
$r=\sum_{i}a_{i}\otimes b_{i}\in A\otimes A$. Define a linear map
$\delta_{r}:A\rightarrow A\otimes A$ by 
\begin{equation}\label{CB3}
		\delta_{r}(x)=-(L_{\ast}(x)\otimes I+I\otimes \mathrm{ad}(x))r,\;\forall x\in A.\end{equation}
Then $(A,\ast,\delta_{r})$ is a pre-Lie bialgebra if and only if the following equations hold:
\begin{align}&\label{Pa6}(L_\ast(x) \otimes I\otimes I+I\otimes L_\ast(x) \otimes I+I \otimes \otimes I\otimes \mathrm{ad}(x))[[r,r]]=0,
\\&\label{Pa7}(P(x\ast y)-P(x)P(y)(r-\tau(r))=0,
\end{align}
where $P(x)=L_\ast(x) \otimes I+I\otimes L_\ast(x)$ and
$[[r,r]]=r_{13}\ast r_{12}+[r_{23},r_{12}]-r_{23}\ast r_{21}-[r_{13},r_{21}]+[r_{23},r_{13}]$.
\end{thm}

\begin{pro} \label{Cg4}
Let $(A,\ast)$ be a pre-Lie algebra and $r=\sum\limits_{i}a_{i}\otimes b_{i}\in A\otimes A$ symmetric.
Assume that $\delta_{r}:A\longrightarrow A\otimes A$ is given
by Eq.~(\ref{CB3}).
Then $(A, \ast,\delta_{r})$ 
 is a pre-Lie bialgebra if and only if
  the following equations hold:
\begin{align}\label{Ly1}(L_\ast(x) \otimes I\otimes I+I\otimes L_\ast(x) \otimes I+I \otimes \otimes I\otimes \mathrm{ad}(x))S(r)=0,
\end{align}
where \begin{equation*}
S(r)=r_{12}\ast r_{23}-r_{12}\ast r_{13}+[r_{13}, r_{23}]=0.\end{equation*}
\end{pro}

Let $(A,\ast)$ be a pre-Lie algebra and $r\in A\otimes A$. The equation 
$S(r)=0$ is called the S-equation in $(A,\ast)$.
 
\begin{thm} \label{Pb1}
	Let $(A,\ast)$ be a pre-Lie algebra and $r\in A\otimes A$. If $r$ is a symmetric solution of 
the S-equation in $(A,\ast)$, then $(A,\ast,\delta_r)$ is a pre-Lie bialgebra. 
\end{thm}
 
 \subsection{Noncommutative pre-Poisson bialgebras}
This section develops the bialgebra theory for noncommutative pre-Poisson algebras.
\begin{defi}  A {\bf noncommutative pre-Poisson coalgebra} is a quadruple
$(A,\Delta_{\succ},\Delta_{\prec},\delta)$, where $(A,\Delta_{\succ},\Delta_{\prec})$ is a dendriform coalgebra
and $(A,\delta)$ is a pre-Lie coalgebra and $\Delta_{\succ},\Delta_{\prec},\delta$ are compatible in the following sense:
\begin{small}
\begin{align}\label{Pc1}&(\delta\otimes I)\Delta_{\succ}-
(\tau\otimes I)(\delta\otimes I)\Delta_{\succ}=(I\otimes \Delta_{\succ})\delta
-(\tau\otimes I)(I\otimes \delta)\Delta_{\succ},\\&
\label{Pc2}(I\otimes \delta)\Delta_{\prec}-
(I\otimes \tau)(I\otimes \delta)\Delta_{\prec}=(\tau\otimes I)
(I\otimes \Delta_{\prec})\delta
-(\tau\otimes I)(\delta\otimes I)\Delta_{\prec}
\\&\label{Pc3}(\Delta_{\succ}\otimes I)\delta+(\Delta_{\prec}\otimes I)\delta=(I\otimes \tau)(\delta\otimes I)\Delta_{\prec}+
(I\otimes \delta)\Delta_{\succ}
.\end{align}
\end{small}
\end{defi}

\begin{defi} \label{Ac1} A noncommutative pre-Poisson coalgebra
$(A,\Delta_{\succ},\Delta_{\prec},\delta)$ is called {\bf coherent} if 
it satisfies
\begin{align}\label{Pc4}(\Delta_{\succ}\otimes I+\Delta_{\prec}\otimes I)\delta=(I\otimes \Delta_{\succ})\delta+
(\tau\otimes I)(I\otimes \tau)(I\otimes \Delta_{\prec})\delta.
\end{align}
\end{defi}

\begin{defi}\label{Ac2} Let  $(A,\succ,\prec,\ast)$ be a coherent noncommutative pre-Poisson algebra.
Suppose that there are comultiplications
$\Delta_{\succ},\Delta_{\prec},\delta:A\longrightarrow A\otimes A$ such that
$(A,\Delta_{\succ},\Delta_{\prec},\delta)$ is a coherent noncommutative pre-Poisson coalgebra. 
If in addition, $(A,\succ,\prec,\Delta_{\succ},\Delta_{\prec})$ is a dendriform bialgebra, $(A,\ast, \delta)$ is
a pre-Lie bialgebra and $\Delta_{\succ},\Delta_{\prec},\delta$ satisfy the following
compatible conditions:
\begin{small}
\begin{align}&\label{NPB1}
	\Delta(x\ast y)=( L_{\ast}(x)\otimes I)\Delta(y)+(I\otimes L_{\ast}(x))\Delta(y) -
(I \otimes L_{\prec}(y))\delta(x)-(R_{\succ}(y)\otimes I)\tau\delta(x),\\&
\label{NPB2}(\delta-\tau\delta)(y\prec x)=(I \otimes R_{\prec}(x))	(\delta-\tau\delta)(y)+(R_{\ast}(y)\otimes I)\tau\Delta_{\succ}(x)
+(I \otimes L_{\prec}(y))\delta(x)-( L_{\ast}(x)\otimes I)\Delta(y),\\&
\label{NPB3}\delta(x\cdot y)=(I\otimes R_{\cdot}(y))\delta(x)+(I\otimes L_{\cdot}(x))\delta(y)-
( L_{\ast}(x)\otimes I)\tau\Delta_{\prec}(y)-( L_{\ast}(y)\otimes I)\Delta_{\succ}(x),\\&
\label{NPB4}\Delta_{\prec}(x\ast y-y\ast x)=(\mathrm{ad}(x)\otimes I)\Delta_{\prec}(y)+(I\otimes  L_{\succ}(y))\tau\delta(x)
+(I\otimes L_{\ast}(x))\Delta_{\prec}(y)-(R_{\cdot}(y)\otimes I)\tau \delta(x),\end{align}
\end{small}
for all $x,y\in A$, then $(A,\succ,\prec,\ast,\Delta_{\succ},\Delta_{\prec},\delta)$ is called
a {\bf noncommutative pre-Poisson bialgebra}.\end{defi}

\begin{rmk} \label {Bc} $(A,\Delta_{\succ},\Delta_{\prec},\delta)$ is 
a (coherent) noncommutative pre-Poisson coalgebra 
 if and only if $(A^{*},\succ_{A^{*}},\prec_{A^{*}},\ast_{A^{*}})$
 is a (coherent) noncommutative pre-Poisson algebra, where $\succ_{A^{*}},\prec_{A^{*}},\ast_{A^{*}}$ 
 are the linear dual of $\Delta_{\succ},\Delta_{\prec},\delta$
 respectively, that is,
 \begin{align}&\label{Dm1}\langle \Delta_{\succ}(x),\zeta\otimes \eta\rangle=\langle x,\zeta\succ_{A^{*}} \eta\rangle, \ \ \
 \langle \Delta_{\prec}(x),\zeta\otimes \eta\rangle=\langle x,\zeta\prec_{A^{*}} \eta\rangle,
\\&\label{Dm3}\langle \delta(x),\zeta\otimes \eta\rangle=\langle x,\zeta\ast_{A^{*}} \eta\rangle,~\forall~x\in A,\zeta,\eta\in A^{*}.
\end{align}
 Thus, a noncommutative pre-Poisson bialgebra $(A,\succ,\prec,\ast,\Delta_{\succ},\Delta_{\prec},\delta)$ is sometimes
denoted by $(A,\succ,\prec,\ast,A^{*},\succ_{A^{*}},\prec_{A^{*}},\ast_{A^{*}})$, where the coherent
noncommutative pre-Poisson
 algebra structure $(\succ_{A^{*}},\prec_{A^{*}},\ast_{A^{*}})$
 on the dual space $A^{*}$
corresponds to the coherent noncommutative pre-Poisson coalgebra $(A,\Delta_{\succ},\Delta_{\prec},\delta)$
through Eqs.~(\ref{Dm1})-(\ref{Dm3}).
\end{rmk}

\begin{thm} \label{Mb3} Let $(A, \succ_1,\prec_1,\ast_1)$ be a coherent noncommutative pre-Poisson algebra
equipped with comultiplications $\Delta_{\succ},\Delta_{\prec},\delta:A\longrightarrow
A\otimes A$. Suppose that $\Delta_{\succ}^{*},\Delta_{\prec}^{*},\delta^{*}:A^{*}\otimes
A^{*}\longrightarrow A^{*}$ induce a coherent noncommutative pre-Poisson algebra structure
on $A^{*}$. Put $\succ_2=\Delta_{\succ}^{*},\prec_2=\Delta_{\prec}^{*},\ast_2=\delta^{*}$. Then the
following conditions are equivalent:
\begin{enumerate}
	\item\label{Ea1} $(A\oplus A^{*},\cdot,[ \ , \ ],\omega)$ is a phase space of the noncommutative Poisson algebra $(A,\cdot_1, [ \ , \ ]_1)$.
\item\label{Ea2} $(A\oplus A^{*},\omega,A,A^{*})$ is a Manin triple of noncommutative pre-Poisson algebras, 
where the skew-symmetric bilinear form $\omega$ on $A\oplus A^{*}$ is given by Eq.~(\ref{C2}).
\item\label{Ea3}
$(A, A^{*}, R_{\cdot_1}^{*},-L_{\prec_1}^{*},-R_{\succ_1}^{*},L_{\cdot_1}^{*},-\mathrm{ad_1}^{*},R^{*}_{\ast_1},
 R_{\cdot_2}^{*},-L_{\prec_2}^{*},-R_{\succ_2}^{*},L_{\cdot_2}^{*},-\mathrm{ad_2}^{*},R^{*}_{\ast_2})$ 
 is a matched pair of coherent noncommutative pre-Poisson algebras.
 \item\label{Ea4}
$(A, A^{*},R_{\prec_1}^{*}, L_{\succ_1}^{*},-L_{\ast_1}^{*},R_{\prec_2}^{*},L_{\succ_2}^{*},-L_{\ast_2}^{*}) $ is a matched pair of coherent 
noncommutative Poisson algebras.
	\item\label{Ea5} $(A, \succ_1,\prec_1,\ast_1,\Delta_{\succ},\Delta_{\prec},\delta)$ is a noncommutative pre-Poisson bialgebra.
\end{enumerate}
\end{thm}

\begin{proof} According to Theorem \ref{Mp1}, we only need to show that \eqref{Ea4} $\Longleftrightarrow$ \eqref{Ea5}, 
that is,
 $(\ref{MP1})\Longleftrightarrow (\ref{NPB1}),~(\ref{MP2})\Longleftrightarrow (\ref{NPB2}),
~(\ref{MP3})\Longleftrightarrow (\ref{NPB3}),~(\ref{MP4})\Longleftrightarrow (\ref{NPB4})$
under the conditions
 \begin{align*}l_1=R_{\prec_1}^{*}, \ \ \ r_1=L_{\succ_1}^{*},\ \ \ \rho_1=-L_{\ast_1}^{*},\ \ \
 l_2=R_{\prec_2}^{*},\ \ \ r_2=L_{\succ_2}^{*},\ \ \ \rho_2=-L_{\ast_2}^{*}.\end{align*}
In fact, we obtain for all $x,y\in A,a,b\in A^{*}$, 
\begin{align*}
\langle -R_{\prec_1}^{*}(L_{\ast_2}^{*}(a)x)b,y\rangle
&=-\langle b,y\prec_1 (L_{\ast_2}^{*}(a)x)\rangle
=-\langle L_{\prec_1}^{*}(y)b, L_{\ast_2}^{*}(a)x\rangle
\\&=-\langle a\ast_2(L_{\prec_1}^{*}(y)b), x\rangle
=-\langle (I\otimes L_{\prec_1}(y))\delta(x),a\otimes b\rangle,\end{align*}
\begin{align*}
\langle -L_{\succ_1}^{*}(L_{\ast_2}^{*}(b)x)a,y\rangle
&=-\langle a,(L_{\ast_2}^{*}(b)x)\succ_1 y\rangle
=-\langle R_{\succ_1}^{*}(y)a, L_{\ast_2}^{*}(b)x\rangle
\\&=-\langle b\ast_2(R_{\succ_1}^{*}(y)a), x\rangle
=-\langle (R_{\succ_1}(y)\otimes I )\tau\delta(x),a\otimes b\rangle,\end{align*}
\begin{align*}\langle -(L_{\ast_1}^{*}(x)a)\cdot_2 b,y
\rangle =-\langle (L_{\ast_1}(x)\otimes I)\Delta(y),a\otimes b\rangle,
\end{align*}
\begin{align*}&
\langle -a\cdot_2 L_{\ast_1}^{*}(x)b,y\rangle=-\langle (I\otimes L_{\ast_1}(x))\Delta(y),a\otimes b\rangle,
\end{align*}
\begin{align*}&
\langle -L_{\ast_1}^{*}(x)(a\cdot_2 b),y\rangle=-\langle a\cdot_2 b,x\ast_{1}y\rangle
=-\langle \delta(x\ast_{1}y),a\otimes b\rangle.\end{align*}
Thus, $(\ref{MP1})\Longleftrightarrow (\ref{NPB1})$. The other equivalence can be proved analogously.
The proof is completed.
\end{proof}

Let $(A, \succ_A,\prec_A,\ast_A,\Delta_{\succ},\Delta_{\prec},\delta)$ be a noncommutative pre-Poisson bialgebra. 
By Proposition \ref{Ps4} and Theorem \ref{Mb3}, we know that $(D=A\oplus A^{*},\succ_{D},\prec_{D},\ast_{D})$
 is a coherent noncommutative pre-Poisson algebra, where
 \begin{align}\label{Db1}(x+a)\succ_{D}(y+b)=&x\succ_A y+( R_{\succ_{A^*}}^{*}+R_{\prec_{A^*}}^{*})
(a)y-L_{\prec_{A^*}}^{*}
(b)x\\&+a\succ_{A^*}b+(R_{\succ_A}^{*}+R_{\prec_A}^{*})(x)b-L_{\prec_A}^{*}(y)a\nonumber
,\end{align}
 \begin{align}\label{Db2}(x+a)\prec_{D}(y+b)=&x\prec_A y-R_{\succ_{A^*}}^{*}
(a)y+(L_{\succ_{A^*}}^{*}+L_{\prec_{A^*}}^{*})
(b)x\\&+a\prec_{A^*}b-R_{\succ_A}^{*}(x)b+(L_{\succ_A}^{*}+L_{\prec_A}^{*})(y)a\nonumber
,\end{align}
 \begin{align}\label{Db3}(x+a)\ast_{D}(y+b)=&x\ast_A y+(R^{*}_{\ast_{A^*}}-L_{\ast_{A^*}}^{*})
(a)y+R^{*}_{\ast_{A^*}}
(b)x\\&+a\ast_{A^*}b+(R^{*}_{\ast_A}-L_{\ast_A}^{*})(x)b+R^{*}_{\ast_A}(y)a\nonumber
,\end{align}
for all $x,y\in A,a,b\in A^{*}$.
$(D=A\oplus A^{*},\succ_{D},\prec_{D},\ast_{D})$ is called the double noncommutative pre-Poisson algebra.


\subsection{Coboundary noncommutative pre-Poisson bialgebras and the noncommutative pre-Poisson Yang-Baxter equation}

In this section, we consider the coboundary noncommutative pre-Poisson
bialgebras. 

\begin{defi} A noncommutative pre-Poisson bialgebra $(A, \succ,\prec,\ast,\Delta_{\succ,r},\Delta_{\prec,r},\delta_{r})$ 
is called coboundary if there is some
$r\in A\otimes A$ such that Eqs.~\eqref{CB1}-\eqref{CB2} and \eqref{CB3} hold.
\end{defi}

\begin{pro} \label{Ac3}
Let $(A,\succ,\prec,\ast,\Delta_{\succ,r},\Delta_{\prec,r},\delta_{r})$ be a noncommutative pre-Poisson bialgebra
and $r=\sum\limits_{i}a_{i}\otimes b_{i}\in A\otimes A$, where
$\Delta_{\succ,r},\Delta_{\prec,r},\delta_r:A\longrightarrow A\otimes A$ are given by Eqs.~(\ref{CB1})-(\ref{CB2}) and
(\ref{CB3}) respectively. Then
\begin{enumerate}
\item \label{CNPB1} Eq.~(\ref{Pc1}) holds if and only if
\begin{small}
\begin{align}\label{CA1}&
(L_{\ast}(x)\otimes I \otimes I)(r_{32}\prec r_{12}-r_{13}\cdot r_{32}+r_{21}\succ r_{31})
\\&-(I\otimes R_{\prec}(x)\otimes I)(r_{12}\ast r_{13}+[r_{32},r_{13}]-r_{12}\ast r_{32})
\nonumber\\&+(I\otimes I \otimes L_{\cdot}(x))(r_{31}\ast r_{12}+[r_{32},r_{12}]-r_{32}\ast r_{21}-[r_{31},r_{21}]+[r_{32},r_{13}])
\nonumber\\&+\sum_{i}(L_{\ast}(b_i\prec x)\otimes I \otimes I +I \otimes L_{\ast}(b_i\prec x)\otimes I)((\tau(r)-r)\otimes a_i)
\nonumber\\&+(L_{\ast}( x\cdot a_i )\otimes I \otimes I)(\tau\otimes I)(b_i\otimes r) -
(L_{\ast}( x\cdot b_i )\otimes I \otimes I)(\tau\otimes I)(a_i\otimes \tau(r))
=0.\nonumber\end{align}\end{small}

\item \label{CNPB2}
	 Eq.~(\ref{Pc2}) holds if and only if
\begin{small}
\begin{align}\label{CA2}&
(R_{\cdot}(x)\otimes I\otimes I)(r_{12}\ast r_{23}+[r_{13},r_{23}]-r_{13}\ast r_{32}+[r_{32},r_{12}]+[r_{13},r_{21}])
\\&+(I \otimes L_{\ast}(x)\otimes I)(r_{32}\succ r_{13}-r_{13}\cdot r_{21}+r_{21}\prec r_{23})
\nonumber\\&+(I \otimes I\otimes L_{\succ}(x))(r_{23}\ast r_{13}-r_{23}\ast r_{21}-[r_{13}, r_{21}])
\nonumber\\&+\sum_{i}(I \otimes L_{\ast}( x\succ b_i)\otimes I+I \otimes I\otimes  L_{\ast}( x\succ b_i))(a_i\otimes(\tau(r)-r))
\nonumber\\&+(I \otimes R_{\ast}( x\succ a_i)\otimes I)(b_i\otimes r)-(I \otimes R_{\ast}( x\succ b_i)\otimes I)(a_i\otimes \tau(r))
=0.\nonumber
\end{align}\end{small}

\item \label{CNPB3} Eq.~(\ref{Pc3}) holds if and only if
\begin{small}
\begin{align}\label{CA3}&
(R_{\prec}(x)\otimes I\otimes I)(r_{21}\ast r_{23}+[r_{31},r_{23}]-r_{12}\ast r_{13})
\\&+(I \otimes L_{\succ}(x)\otimes I)(r_{12}\ast r_{13}+[r_{32},r_{13}]-r_{21}\ast r_{23})
\nonumber\\&+(I \otimes I\otimes \mathrm{ad}(x))(r_{23}\cdot r_{21}-r_{21}\prec r_{13}-r_{23}\succ r_{12}+r_{12}\cdot r_{13}-r_{13}\cdot r_{32})
\nonumber\\&+\sum_{i}(R_{\prec}(x\ast a_i)\otimes I \otimes I)((r-\tau(r))\otimes b_i)
+(I \otimes L_{\succ}(x\ast a_i)\otimes I)((\tau(r)-r)\otimes b_i)
\nonumber\\&+(I \otimes I\otimes \mathrm{ad}(x\cdot b_i))(a_i\otimes \tau(r))-(I \otimes I\otimes \mathrm{ad}(x\cdot a_i))(b_i\otimes r)
=0.\nonumber
\end{align}\end{small}

\item \label{CNPB4} Eq.~(\ref{Pc4}) holds if and only if
\begin{small}
\begin{align}\label{CA4}&
(L_{\ast}(x)\otimes I \otimes I)(r_{32}\prec r_{12}-r_{13}\cdot r_{32}+r_{21}\succ r_{31})
\\&+(I \otimes L_{\ast}(x)\otimes I)(r_{21}\succ r_{31}-r_{31}\cdot r_{23}+r_{32}\prec r_{12})
\nonumber\\&+(I \otimes I\otimes \mathrm{ad}(x))(r_{23}\cdot r_{21}-r_{21}\prec r_{13}-r_{23}\succ r_{12}+r_{12}\cdot r_{13}-r_{13}\cdot r_{32})
\nonumber\\&+\sum_{i}(R_{\prec}(x\ast a_i)\otimes I \otimes I)((r-\tau(r)\otimes b_i)
\nonumber\\&+(R_{\succ}(x\ast a_i)\otimes I \otimes I)(r\otimes b_i)-(R_{\succ}(x\ast b_i)\otimes I \otimes I)(\tau(r)\otimes a_i)
\nonumber\\&+( I \otimes L_{\prec}(x\ast a_i)\otimes I)(\tau(r)\otimes b_i)-
( I \otimes L_{\prec}(x\ast b_i)\otimes I)(r\otimes a_i)
\nonumber\\&+(I \otimes I\otimes R_{\cdot}([x,a_i])(\tau\otimes I)(b_i\otimes r)
-(I \otimes I\otimes R_{\cdot}([x,b_i])(\tau\otimes I)(a_i\otimes \tau(r))
=0.\nonumber
\end{align}\end{small}

\item \label{CNPB5} Eq.~(\ref{NPB1}) holds if and only if
\begin{small}
	\begin{align}\label{CA5}
		&(I\otimes L_{\succ}(x\ast y)-I\otimes L_{\ast}(x)L_{\succ}( y)-I\otimes L_{\prec}(y)\mathrm{ad}(x)
+L_{\ast}(x)R_{\cdot}(y)\otimes I\\&-R_{\cdot}(x\ast y)\otimes I+R_{\cdot}(y)\otimes L_{\ast}(x)-L_{\ast}(x)\otimes
L_{\cdot}( y))(r-\tau(r))=0.\nonumber\end{align}
\item \label{CNPB6} Eq.~(\ref{NPB2}) holds if and only if
		\begin{align}\label{CA6}&
			(I\otimes R_{\prec}(x)L_{\ast}(y)-I\otimes L_{\ast}(y\prec x)-\mathrm{ad}(y\prec x) \otimes I
-L_{\ast}(x)R_{\prec}(y)\otimes I\\&+\mathrm{ad}(y) \otimes R_{\prec}(x)+L_{\ast}(x)\otimes
L_{\cdot}( y))(r-\tau(r))=0.\nonumber\end{align}

\item \label{CNPB6} Eq.~(\ref{NPB3}) holds if and only if
		\begin{align}\label{CA7}
			(L_{\ast}(x)\otimes R_{\cdot}(y)+L_{\ast}(y)\otimes L_{\cdot}(y)-L_{\ast}(x\cdot y)\otimes I)(r-\tau(r))=0.\end{align}
\item \label{CNPB6} Eq.~(\ref{NPB4}) holds if and only if
		\begin{align}\label{CA8}
			(R_{\cdot}(y)\otimes L_{\ast}(x)-\mathrm{ad}(x)\otimes L_{\succ}(y)+\mathrm{ad}(x)R_{\cdot}(y)\otimes I-
R_{\cdot}([x,y])\otimes I-I\otimes L_{\succ}(y) L_{\ast}(x))(r-\tau(r))=0.\end{align}
\end{small}
\end{enumerate}
\end{pro}

\begin{proof}
Combining Eqs.~\eqref{CB1}-\eqref{CB2} and \eqref{CB3}, we have for all $x\in A$, 
\begin{small}
\begin{align*}&(\delta_{r}\otimes I)\Delta_{\succ,r}(x)-
(\tau\otimes I)(\delta_{r}\otimes I)\Delta_{\succ,r}(x)-(I\otimes \Delta_{\succ,r})\delta_{r}(x)
+(\tau\otimes I)(I\otimes \delta)\Delta_{\succ,r}(x)\\=&
\sum_{i,j}b_i\ast a_j\otimes b_j\otimes x\cdot a_i+a_j\otimes [b_i,b_j]\otimes x\cdot a_i-(b_i\prec x)\ast a_j\otimes b_j\otimes a_i
-a_j\otimes [b_i\prec x,b_j]\otimes a_i\\&
+b_j\otimes (b_i\prec x)\ast a_j\otimes a_i+ [b_i\prec x,b_j]\otimes a_j\otimes a_i
-b_j\otimes b_i\ast a_j\otimes  x\cdot a_i-[b_i,b_j]\otimes a_j\otimes  x\cdot a_i
\\&+ x\ast a_i\otimes b_j\prec b_i\otimes a_j-x\ast a_i\otimes b_j\otimes b_i\cdot a_j
+a_i\otimes b_j\prec [x,b_i]\otimes a_j-a_i\otimes b_j\otimes[x,b_i]\cdot a_j
\\&-a_i\ast a_j\otimes b_i\prec x\otimes b_j-a_j\otimes b_i\prec x\otimes [a_i,b_j]
+(x\cdot a_i)\ast a_j\otimes b_i\otimes b_j+a_j\otimes b_i \otimes [x\cdot a_i,b_j]\\=&A(1)+A(2)+A(3),
\end{align*}
\end{small}
where
\begin{small}
\begin{align*}A(1)=&\sum_{i,j}-(b_i\prec x)\ast a_j\otimes b_j\otimes a_i
+ [b_i\prec x,b_j]\otimes a_j\otimes a_i
\\&+ x\ast a_i\otimes b_j\prec b_i\otimes a_j-x\ast a_i\otimes b_j\otimes b_i\cdot a_j
+(x\cdot a_i)\ast a_j\otimes b_i\otimes b_j,
\end{align*}
\begin{align*}A(2)=&\sum_{i,j}-a_j\otimes [b_i\prec x,b_j]\otimes a_i
+b_j\otimes (b_i\prec x)\ast a_j\otimes a_i
\\&+a_i\otimes b_j\prec [x,b_i]\otimes a_j
-a_i\ast a_j\otimes b_i\prec x\otimes b_j-a_j\otimes b_i\prec x\otimes [a_i,b_j],\end{align*}
\begin{align*}A(3)=&\sum_{i,j}b_i\ast a_j\otimes b_j\otimes x\cdot a_i+a_j\otimes [b_i,b_j]\otimes x\cdot a_i
-b_j\otimes b_i\ast a_j\otimes  x\cdot a_i\\&-[b_i,b_j]\otimes a_j\otimes  x\cdot a_i
-a_i\otimes b_j\otimes[x,b_i]\cdot a_j
+a_j\otimes b_i \otimes [x\cdot a_i,b_j].\end{align*}\end{small}
Using Eqs.~(\ref{Np1})-(\ref{Np4}) and (\ref{Np0}), we have
\begin{small}
\begin{align*}A(1)=&(L_{\ast}(b_i\prec x)\otimes I\otimes I)((\tau(r)-r)\otimes a_i)
-b_j\ast(b_i\prec x)\otimes a_j\otimes a_i\\&+(L_{\ast}(x)\otimes I\otimes I)(r_{32}\prec r_{12}-r_{13}\cdot r_{32})
+a_i\ast (a_j\prec x)\otimes b_i\otimes b_j+x\ast(a_i\succ a_j)\otimes b_i\otimes b_j
\\=&(L_{\ast}(b_i\prec x)\otimes I\otimes I)((\tau(r)-r)\otimes a_i)
-(R_{\ast}(b_i\prec x)\otimes I\otimes I)(\tau(r)\otimes a_i)
\\&+(L_{\ast}(x)\otimes I\otimes I)(r_{32}\prec r_{12}-r_{13}\cdot r_{32}+r_{12}\succ r_{13})
+(R_{\ast}(a_i\prec x)\otimes I\otimes I)(r\otimes b_i)
\end{align*}
\begin{align*}A(2)=&\sum_{i,j}a_j\otimes b_j\ast (b_i\prec x)\otimes a_i+(I\otimes L_{\ast}(b_i\prec x) \otimes I)((\tau(r)-r)\otimes a_i)
+a_i\otimes(b_i\ast b_j)\otimes \otimes a_j\\&-a_i\otimes b_j\prec [x,b_i]\otimes a_j
-(I\otimes R_{\prec}(x)\otimes I)(r_{12}\ast r_{13}+[r_{32},r_{13}])
\\=&(I\otimes L_{\ast}(b_i\prec x) \otimes I)((\tau(r)-r)\otimes a_i)
-(I\otimes R_{\prec}(x)\otimes I)(r_{12}\ast r_{13}+[r_{32},r_{13}]-r_{12}\ast r_{32})
,\end{align*}
\begin{align*}A(3)=(I\otimes I\otimes L_{\cdot}(x))(r_{31}\ast r_{12}+[r_{32},r_{12}]-r_{32}\ast r_{21}-[r_{31},r_{21}]
-[r_{13},r_{32}])
.\end{align*}\end{small}
Therefore, Eq.~(\ref{Pc1}) holds if and only if Eq.~(\ref{CA1}) holds. The remaining Items can be checked in the same way. \end{proof}

\begin{thm} \label{Bc1} Let $(A,\succ,\prec,\ast)$ be a coherent noncommutative pre-Poisson algebra equipped with
linear maps $\Delta_{\succ,r},\Delta_{\prec,r},\delta_r:A\longrightarrow A\otimes A$ given 
by Eqs.~\eqref{CB1}-\eqref{CB2} and \eqref{CB3} respectively. 
Then $(A, \succ,\prec,\ast,\Delta_{\succ,r},\Delta_{\prec,r},\delta_r)$ 
 is a noncommutative pre-Poisson bialgebra if and only if 
 Eqs.~\eqref{Pa1}-\eqref{Pa5}, \eqref{Pa6}-\eqref{Pa7} and
 \eqref{CA1}-\eqref{CA8} hold.\end{thm}
\begin{proof} This follows from Theorem \ref{Cg1}, Theorem \ref{Cg3} and Proposition \ref{Ac3}.\end{proof}

Denote 
\begin{align*}&D_{1}(r)=r_{23}\succ r_{13}-r_{13}\cdot r_{12}+r_{12}\prec r_{23},\\&
D_{2}(r)=r_{13}\cdot r_{23}-r_{12}\succ r_{13}-r_{23}\prec r_{12},
\\&S_1(r)=r_{23}\ast r_{12}+[r_{13},r_{12}]-r_{23}\ast r_{13}.\end{align*}

The following conclusion is apparent.

\begin{pro} \label{Ac4}
Let $(A,\succ,\prec,\ast)$ be a coherent noncommutative pre-Poisson algebra and $r=\sum\limits_{i}a_{i}\otimes b_{i}\in A\otimes A$ symmetric.
Define linear maps $\Delta_{\succ,r},\Delta_{\prec,r},\delta_r:A\longrightarrow A\otimes A$ 
by Eqs.~(\ref{CB1})-(\ref{CB2}) and (\ref{CB3}) respectively
Then $(A, \succ,\prec,\ast,\Delta_{\succ,r},\Delta_{\prec,r},\delta_r)$. 
 is a noncommutative pre-Poisson bialgebra if and only if
 Eqs.~(\ref{Dy1})-(\ref{Dy3}), (\ref{Ly1}) and
  the following equations hold:
\begin{align}&\label{Py1}
(L_{\ast}(x)\otimes I \otimes I)D_{2}(r)
+(I\otimes R_{\prec}(x)\otimes I-I\otimes I \otimes L_{\cdot}(x))S(r)=0,
\\&\label{Py2}
(R_{\cdot}(x)\otimes I\otimes I-I \otimes I\otimes L_{\succ}(x))S_1(r)
+(I \otimes L_{\ast}(x)\otimes I)D_{1}(r)=0,
\\&\label{Py3}
(R_{\prec}(x)\otimes I\otimes I-I \otimes L_{\succ}(x)\otimes I)S(r)
-(I \otimes I\otimes \mathrm{ad}(x))D_{2}(r)=0,
\\&\label{Py4}
(L_{\ast}(x)\otimes I \otimes I+I \otimes L_{\ast}(x)\otimes I+I \otimes I\otimes \mathrm{ad}(x))D_{2}(r)
=0.
\end{align}

\end{pro}

\begin{rmk} \label{Ac6} Let $\sigma_{132},\sigma_{23},\sigma_{13}:A\otimes A\otimes A\longrightarrow A\otimes A\otimes A$ 
be linear maps defined by
\begin{equation*}\sigma_{132}(x\otimes y\otimes z)=z\otimes x\otimes y, \ \ \
\sigma_{23}(x\otimes y\otimes z)=x\otimes z\otimes y, \ \ \
\sigma_{13}(x\otimes y\otimes z)=z\otimes y\otimes x \end{equation*}
for all $x,y,z\in A$. If $r$ is symmetric, then
\begin{align*}&\sigma_{13}S(r)=S_{1}(r),\ \ \  \sigma_{23}D(r)=-D_{1}(r),\ \ \  \sigma_{132}D(r)=D_{2}(r).\end{align*}
\end{rmk}

\begin{thm} \label{Ac5} Let $(A,\succ,\prec,\ast)$ be a coherent noncommutative pre-Poisson algebra
and $r\in A\otimes A$ be symmetric.
Then $(A, \succ,\prec,\ast,\Delta_{\succ,r},\Delta_{\prec,r},\delta_r)$ is a noncommutative pre-Poisson bialgebra if
$S(r)=D(r)=0$, where $\Delta_{\succ,r},\Delta_{\prec,r},\delta_r$ are defined by Eqs.~\eqref{CB1}-\eqref{CB2} and \eqref{CB3} respectively.
\end{thm}
\begin{proof} 
The statement follows from Theorem \ref{Zb1}, Theorem \ref{Pb1}, Theorem \ref{Bc1}, Proposition \ref{Ac4} and Remark \ref{Ac6}.
\end{proof}

\begin{defi} 
Let $(A,\succ,\prec,\ast)$ be a coherent noncommutative pre-Poisson algebra and $r\in
A\otimes A$. We say that  $r$ satisfies {\bf the noncommutative-pre-Poisson Yang-Baxter equation} or {\bf NPP-YBE} in short
 if $r$ satisfies both the D-equation 
\begin{equation*}D(r)=r_{12}\cdot r_{13}-r_{23}\succ r_{12}-r_{13}\prec r_{23}=0\end{equation*}
 and the S-equation:
\begin{equation*} S(r)=r_{12}\ast r_{23}-r_{12}\ast r_{13}+[r_{13}, r_{23}]=0,\end{equation*}
where $x\cdot y=x\succ y+x\prec y$ and $[x,y]=x\ast y-y\ast x$.
 \end{defi}
 
  \begin{ex} Let $(A, \succ, \prec, \ast)$ be
 the 3-dimensional coherent noncommutative pre-Poisson algebra 
  defined in Example \ref{Ae1}, whose basis  is${e_1, e_2, e_3}$.
   By direct calculation, one finds that the tensor $r = e_2 \otimes e_2$ is a symmetric solution of 
   the noncommutative pre-Poisson Yang-Baxter equation in $(A, \succ, \prec, \ast)$.
  \end{ex}
 
Consequently, we arrive at the following conclusion.
 
\begin{cor}\label{Bc2} 
	Let $(A,\succ,\prec,\ast)$ be a coherent noncommutative pre-Poisson algebra and $r\in A\otimes A$ be
	a symmetric solution of the NPP-YBE in $(A,\succ,\prec,\ast)$.
	Then $(A, \succ,\prec,\ast,\Delta_{\succ,r},\Delta_{\prec,r},\delta_r)$
 is a noncommutative pre-Poisson bialgebra, where $\Delta_{\succ,r},\Delta_{\prec,r},\delta_r:A\rightarrow A\otimes A$ are 
linear maps given by Eqs.~\eqref{CB1}-\eqref{CB2} and \eqref{CB3} respectively.
\end{cor}

For a vector space $A$, the isomorphism $A\otimes A^{*}\simeq Hom (A^{*},A)$ identifies an element $r\in A\otimes A$ with a map
$T_{r}:A^{*}\longrightarrow A$. Explicitly, 
\begin{equation} \label{Op1} T_{r}:A^{*}\longrightarrow A,\ \ \  \langle T_{r}(\zeta),\eta\rangle=\langle r,\zeta\otimes\eta\rangle,
\ \ \ \forall~\zeta,\eta\in A^{*}.\end{equation}
It is clearly that $T_{r}^{*}=T_{\tau(r)},~T_{r-\tau(r)}^{*}=-T_{r-\tau(r)}$.

\begin{defi} Assume that $(A,\succ,\prec,\ast)$ is a noncommutative pre-Poisson algebra and $(V, $\ \ \ \ \ \ \ \ \ \ $ \succ_V,\prec_V,\ast_V,l_{\succ},r_{\succ},l_{\prec},r_{\prec},l_{\ast},r_{\ast})$ is
 an A-noncommutative pre-Poisson algebra.
A relative Rota-Baxter operator $T$ of weight $\lambda$ on $(A,\succ,\prec,\ast)$ associated to
$(V,\succ_V,\prec_V,\ast_V,l_{\succ},r_{\succ},l_{\prec},r_{\prec},l_{\ast},r_{\ast})$
   is a linear map $T:V\longrightarrow A$ satisfying
\begin{align*}&T(u)\succ T(v)=T (l_{\succ}(T(u))v+r_{\succ}(T(v))u+\lambda u\succ_V v),\\&
T(u)\prec T(v)=T (l_{\prec}(T(u))v+r_{\prec}(T(v))u+\lambda u\prec_V v),\\&
T(u)\ast T(v)=T (l_{\ast}(T(u))v+r_{\ast}(T(v))u+\lambda u\ast_V v).
\end{align*}
 When $u\succ_Vv=u\prec_Vv=u\ast_Vv=0$ for all $u,v\in V$, then $T$ is simply a relative Rota-Baxter operator ( $\mathcal O$-operator) on
$(A,\succ,\prec,\ast)$ associated to the representation $(V,l_{\succ},r_{\succ},l_{\prec},r_{\prec},l_{\ast},r_{\ast})$.
\end{defi}

Combining Theorem 4.4.11 \cite{B2} and Theorem 6.6 \cite{B1}, we have the following conclusion.

\begin{thm} \label{Opp} Let $(A,\succ,\prec,\ast)$ be a coherent noncommutative pre-Poisson algebra and
$r\in A\otimes A$ be symmetric. Then the following conditions are equivalent:
\begin{enumerate}
		\item $r$ is a solution of the NPP-YBE in $(A,\succ,\prec,\ast)$.
 \item $T_r$
is an $\mathcal O$-operator of the noncommutative Poisson algebra $(A, \cdot, [ \ , \ ])$ associated to
 $(A^{*}, R_{\prec}^{*}, L_{\succ}^{*},-L_{\ast}^{*})$.
\item $T_r$
is an $\mathcal O$-operator of $(A,\succ,\prec,\ast)$ associated to 
$(A^{*},R_{\cdot}^{*},-L_{\prec}^{*},-R_{\succ}^{*},L_{\cdot}^{*},-\mathrm{ad}^{*},R^{*}_{\ast})$.
\end{enumerate}
\end{thm}

\begin{thm} \label{Yo}
 Let $(V,l_{\succ},r_{\succ},l_{\prec},r_{\prec},l_{\ast},r_{\ast})$ be
  a representation of a coherent noncommutative pre-Poisson algebra satisfying Eqs.~\eqref{r10}-\eqref{r12}.
  Assume that $(V^*,r_{\cdot}^{*},-l_{\prec}^{*},-r_{\succ}^{*},l_{\cdot}^{*},r_{\ast}^{*}-l_{\ast}^{*},r_{\ast}^{*})$
 is the dual representation of $A$ given by Proposition \ref{Dr}. 
 Denote $\hat{A}=A\ltimes V^{*}$ and suppose
that $T:V\longrightarrow A$ is a linear map which is identifies an element in $\hat{A}\otimes \hat{A}$ through
 ($Hom(V,A)\simeq A\otimes V^{*}\subseteq \hat{A}\otimes \hat{A}$).
  Then $r=T+\tau(T)$ is a symmetric solution of
 the NPP-YBE in the coherent noncommutative pre-Poisson algebra $\hat{A}$ if and only if $T$ is an $\mathcal O$-operator
on $(A,\succ,\prec,\ast)$ associated to $(V,l_{\succ},r_{\succ},l_{\prec},r_{\prec},l_{\ast},r_{\ast})$.
\end{thm}

\begin{proof}
Based on Remark \ref{Ms}, if $(V,l_{\succ},r_{\succ},l_{\prec},r_{\prec},l_{\ast},r_{\ast})$ is a representation of
 a coherent noncommutative pre-Poisson algebra $(A,\succ,\prec,\ast)$ satisfying Eqs.~\eqref{r10}-\eqref{r12},
 then $A\ltimes V^{*}$ is a coherent noncommutative pre-Poisson algebra. Combining 
 Theorem 4.4.13 \cite{B2} and Theorem 6.9 \cite{B1}, we obtain the desired result.
\end{proof}

\begin{ex} Let $(A=ke_1\oplus ke_2,\succ,\prec,\ast)$ be the 2-dimensional
 coherent noncommutative 
pre-Poisson algebra given in Example \ref{Ae}. Define a linear map $T:A\longrightarrow A$
by $T(e_1)=e_2,~T(e_2)=0$.
 It is easy to check that $T$ an $\mathcal{O}$-operator of $(A,\succ,\prec,\ast)$ 
 associated to the regular representation
$(A,L_{\succ},R_{\succ},L_{\prec},R_{\prec},L_{\ast},R_{\ast})$.
Denote the dual basis of $A^{*}$ by $\{e_1^{*}, e_2^{*}  \}$. 
By Remark \ref{Ms}, the semi-direct product 
$A\ltimes A^{*}$ with respect to its representation
$(A^{*},R_{\cdot}^{*},-L_{\prec}^{*},-R_{\succ}^{*},L_{\cdot}^{*},-\mathrm{ad}^{*},R^{*}_{\ast})$
is a coherent noncommutative pre-Poisson algebra with the multiplications $(\succ,\prec,\ast)$ defined by 
 \begin{align*}&
 (x+u)\succ (y+v)=x\succ y+R_{\cdot}^{*}(x)v-L_{\prec}^{*}(y)u, \\&
(x+u)\prec (y+v)=x\prec y-R_{\succ}^{*}(x)v+L_{\cdot}^{*}(y)u,\\&
(x+u)\ast (y+v)=x\ast y-\mathrm{ad}^{*}(x)v+R_{\ast}^{*}(y)u,
\ \ ~\forall~x,y\in A,u,v\in A^{*}.
	\end{align*}
 Explicitly, the non-trivial multiplications $(\succ,\prec,\ast)$ are given by 
 \begin{align*} &e_1\succ e_2=-e_1 \prec e_2=e_2\prec e_1=e_2\ast e_1=e_2,\ \ \ \ \ 
 e_1 \succ e_1=e_1\ast e_1=e_1,\\&
e_2^{*}\succ e_2=-e_1\succ e_1^{*}=e_2^{*}\prec e_2=e_1\prec e_1^{*}=e_1^{*}\prec e_1=e_2\ast e_2^{*}=-e_1^{*}\ast e_1=e_1^{*}
\\& e_1\succ e_2^{*}= e_2^{*}\succ e_1= e_1\ast e_2^{*}=e_2^{*}\ast e_1= -e_2^{*},
\end{align*}
In view of Theorem \ref{Yo}, we know that
 $r=\sum_{i =1}^{2}T(e_i)\otimes e_i^{*}+e_i^{*}\otimes T(e_i)$
 is a symmetric solution of the NPP-YBE in the coherent noncommutative pre-Poisson algebra $ A\ltimes A^{*} $. 
 According to Corollary \ref{Bc2}, 
$ (A\ltimes A^{*}, \Delta_{\succ}, \Delta_{\prec},\delta)$ is a noncommutative pre-Poisson bialgebra with the linear maps
$\Delta_{\succ}, \Delta_{\prec},\delta:A\ltimes A^{*}\longrightarrow (A\ltimes A^{*})\otimes (A\ltimes A^{*})$ defined respectively by
 \begin{align*}&\Delta_{\succ}(x)=(I\otimes L_{\cdot}(x)-R_{\prec}(x)\otimes I)\tau(r),\\&
 \Delta_{\prec}(x)=(R_{\cdot}(x)\otimes I-I\otimes L_{\succ}(x))r,\\&
 \delta(x)=-(L_{\ast}(x)\otimes I+I\otimes \mathrm{ad}(x))r,~~\forall~x\in A\ltimes A^{*}.
\end{align*}
Explicitly, the comultiplications $(\Delta_{\succ}, \Delta_{\prec},\delta)$ are given as follows (only non-zero operations are listed): 
 \begin{align*}&
\delta(e_1)=e_1^{*}\otimes e_2-e_2\otimes e_1^{*}, \ \ \ \delta(e_{2}^{*})=e_1^{*}\otimes e_1^{*},
\\&
\Delta_{\succ}(e_1)=-e_2\otimes e_1^{*} -e_1^{*}\otimes e_2 ,\ \ \ \Delta_{\succ}(e_2^{*})=e_1^{*} \otimes e_1^{*},
\\&
  \Delta_{\prec}(e_1)=e_2\otimes e_1^{*}-2e_1^{*}\otimes e_2.
\end{align*}
\end{ex}
\section{Quasi-triangular noncommutative pre-Poisson bialgebras and factorizable noncommutative pre-Poisson bialgebras}
Having shown in Theorem \ref{Ac5} that symmetric solutions of the NPP-YBE
 yield coboundary noncommutative pre-Poisson bialgebras, we now address the more general case of non-symmetric solutions.

\subsection{Quasi-triangular noncommutative pre-Poisson bialgebras}
\begin{defi} \label{In}
 Let $(A,\succ,\prec,\ast)$ be a noncommutative pre-Poisson algebra and $r\in A\otimes A$. Then $r$ is called {\bf invariant} if 
 \begin{align}&\label{Iv1}
(I\otimes L_{\cdot}(x)-R_{\prec}(x)\otimes I)\tau(r)=0,
\\&\label{Iv2}(R_{\cdot}(x)\otimes I-I\otimes L_{\succ}(x))r=0,
\\&\label{Iv3}(L_{\ast}(x)\otimes I+I\otimes \mathrm{ad}(x))r=0,~\forall~x\in A.
\end{align}
\end{defi} 

\begin{lem} \label{In1}
 Let $(A, \succ,\prec,\ast)$ be a noncommutative pre-Poisson algebra and $r\in A\otimes A$. Then $r$ is invariant if and only if
 \begin{align}&
\label{Iv4}R_{\prec}(x) T_{r}(\zeta)=T_{r}(L_{\cdot}^{*}(x)\zeta),\\&
\label{Iv5}L_{\succ}(x) T_{r}(\zeta)=T_{r}(R_{\cdot}^{*}(x)\zeta),
\\&\label{Iv6}\mathrm{ad}(x) T_{r}(\zeta)=-T_{r}(L_{\ast}^{*}(x)\zeta),~~\forall~x\in A,\zeta\in A^{*}.
\end{align}
Moreover, Eqs.~(\ref{Iv1})-(\ref{Iv3}) hold if and only if the following equations hold:
\begin{align}&\label{Iv7}R_{\cdot}^{*}(T_{\tau(r)}(\zeta))\eta=L_{\prec}^{*}(T_{r}(\eta))\zeta,\\&
\label{Iv8}L_{\cdot}^{*} (T_{\tau(r)}(\eta))\zeta=R_{\succ}^{*}(T_{r}(\zeta))\eta
\\&\label{Iv9}\mathrm{ad}^{*} (T_{r}(\zeta))\eta=R_{\ast}^{*}(T_{\tau(r)}(\eta))\zeta,~~\forall~x\in A,\zeta\in A^{*}.
\end{align}
\end{lem}
\begin{proof} Based on Lemma 2.12 \cite{Wa} and Lemma 2.9 \cite{Wbl}, Eqs.~(\ref{Iv4})-(\ref{Iv6}) hold.
For all $~x\in A,\zeta,\eta\in A^{*}$, we have
\begin{small}
\begin{align*}\langle(I\otimes L_{\cdot}(x)-R_{\prec}(x)\otimes I)\tau(r),\zeta\otimes \eta\rangle
=&\langle \tau(r),\zeta\otimes L_{\cdot}^{*}(x)\eta\rangle-\langle \tau(r),R_{\prec}^{*}(x)\zeta\otimes \eta\rangle
\\=&\langle T_{\tau(r)}(\zeta), L_{\cdot}^{*}(x)\eta\rangle-\langle T_{r}(\eta), R_{\prec}^{*}(x)\zeta\rangle
\\=&\langle x\cdot T_{\tau(r)}(\zeta),\eta\rangle-\langle T_{r}(\eta)\prec x, \zeta\rangle
\\=&\langle R_{\cdot}^{*}(T_{\tau(r)}(\zeta))\eta-L_{\prec}^{*}(T_{r}(\eta))\zeta,x\rangle.\end{align*}
\end{small}
Thus, Eqs.~$(\ref{Iv1})\Longleftrightarrow (\ref{Iv4})\Longleftrightarrow (\ref{Iv7})$. By the same token,
Eqs.~$(\ref{Iv2})\Longleftrightarrow (\ref{Iv5})\Longleftrightarrow (\ref{Iv8})$
and Eqs.~$(\ref{Iv3})\Longleftrightarrow (\ref{Iv6})\Longleftrightarrow (\ref{Iv9})$.

\end{proof}

\begin{pro} \label{Si}
 Let $(A,\succ,\prec,\ast)$ be a noncommutative pre-Poisson algebra and $r\in A\otimes A$. Then the following
 conditions are equivalent:
  \begin{enumerate}
\item $r-\tau(r)$ is invariant.
\item The following equations hold:
\begin{small}
\begin{align}&\label{Iv10}R_{\cdot}^{*}(T_{r-\tau(r)}(\zeta))\eta=-L_{\prec}^{*}(T_{r-\tau(r)}(\eta))\zeta,\ \
L_{\cdot}^{*} (T_{r-\tau(r)}(\eta))\zeta=-R_{\succ}^{*}(T_{r-\tau(r)}(\zeta))\eta, \ \
\mathrm{ad}^{*} (T_{r-\tau(r)}(\zeta))\eta=-R_{\ast}^{*}(T_{r-\tau(r)}(\eta))\zeta
.\end{align} \end{small}
\item The following equations hold:
\begin{small}
\begin{align}&\label{Iv11}R_{\prec}(x) T_{r-\tau(r)}(\zeta)=T_{r-\tau(r)}(L_{\cdot}^{*}(x)\zeta), \ \
L_{\succ}(x) T_{r-\tau(r)}(\zeta)=T_{r-\tau(r)}(R_{\cdot}^{*}(x)\zeta),\ \
\mathrm{ad}(x) T_{r-\tau(r)}(\zeta)=-T_{r-\tau(r)}(L_{\ast}^{*}(x)\zeta).\end{align}\end{small}
\item The following equations hold:
\begin{small}
\begin{align}&\label{Iv12} T_{r-\tau(r)}(R_{\prec}^{*}x)(\zeta)=L_{\cdot}(x)T_{r-\tau(r)}(\zeta), \ \
 T_{r-\tau(r)}(L_{\succ}^{*}(x)\zeta)=R_{\cdot}(x)T_{r-\tau(r)}(\zeta),\ \
T_{r-\tau(r)}(\mathrm{ad}^{*}(x) \zeta)=-L_{\ast}(x)T_{r-\tau(r)}(\zeta).\end{align}\end{small}
\end{enumerate}
for all $x\in A,~\zeta,\eta\in A^{*}$
 \end{pro}
  
  \begin{proof} It follows from Lemma 2.13 \cite{Wa} and Proposition 2.12 \cite{Wbl} and Lemma \ref{In1}.
\end{proof}
  
By Eqs.~(\ref{Iv10})-(\ref{Iv12}), we have 
 \begin{align}&\label{Iv13} 
T_{r-\tau(r)}(R_{\succ}^{*}(x)\zeta)=- L_{\prec}(x)T_{r-\tau(r)}(\zeta),\ \ \ T_{r-\tau(r)}(L_{\prec}^{*}(x)\zeta)=- R_{\succ}(x)T_{r-\tau(r)}(\zeta),
 \\&\label{Iv14}T_{r-\tau(r)}(R_{\ast}^{*}(x)\zeta)=R_{\ast}(x)T_{r+\tau(r)}(\zeta),\\&
\label{Iv15} L_{\ast}^{*} (T_{r-\tau(r)}(\zeta))\eta=-L_{\ast}^{*}(T_{r-\tau(r)}(\eta))\zeta, \ \ \ R_{\prec}^{*}(T_{r-\tau(r)}(\zeta))\eta=L_{\succ}^{*}(T_{r-\tau(r)}(\eta))\zeta.
\end{align}

\begin{thm} \label{Qs} Let $(A,\succ,\prec,\ast)$ be a coherent noncommutative pre-Poisson algebra
 and $r=\sum_{i}a_i\otimes b_i\in A\otimes A$. 
 Assume that
$\Delta_{\succ,r},\Delta_{\prec,r},\delta_r$ are given by Eqs.~(\ref{CB1})-(\ref{CB2}) and (\ref{CB3}).
 If $r$ is a solution of the 
NPP-YBE in $(A,\succ,\prec,\ast)$ and $r-\tau(r)$ is invariant.
Then $(A,\succ,\prec,\ast,\Delta_{\succ,r},\Delta_{\prec,r},\delta_r)$ is a noncommutative pre-Poisson bialgebra.
\end{thm}
\begin{proof}
Since $r$ is a solution of the 
NPP-YBE in $(A,\succ,\prec,\ast)$ and $r-\tau(r)$ is invariant, $S(r)=D(r)=0$ and
Eqs.~(\ref{Iv10})-(\ref{Iv15}) hold. By Proposition 2.16 \cite{Wa}, By Proposition 2.10 \cite{Wbl} and Theorem 2.14 \cite{Wbl},
 $(A,\succ,\prec,\Delta_{\succ,r},\Delta_{\prec,r})$ is a dendriform bialgebra
 and $(A,\ast,\delta_{r})$ is a pre-Lie bialgebra.
 Note that
\begin{align*}
&r_{32}\prec r_{12}-r_{13}\cdot r_{32}+r_{21}\succ r_{31}
\\=&-\sigma_{132}D(r)+(r_{31}-r_{13})\cdot r_{32}+(r_{21}-r_{12})\succ r_{31}
\\=&-\sigma_{132}D(r)+\sum_i(I\otimes I\otimes R_{\cdot}(a_i))(\tau\otimes I)(b_i\otimes (\tau(r)-r))
+(R_{\succ}(b_i)\otimes I\otimes I )((\tau(r)-r)\otimes a_i),
\\&r_{12}\ast r_{13}+[r_{32},r_{13}]-r_{12}\ast r_{32}
\\=&-S(r)+r_{12}\ast(r_{23}-r_{32})+[r_{13},r_{23}-r_{32}]
\\=&-S(r)+\sum_i(I\otimes L_{\ast}(b_i)\otimes I)(a_i\otimes(r-\tau(r)))
+(I\otimes I\otimes \mathrm{ad}(b_i))(a_i\otimes(r-\tau(r))),
\\&r_{31}\ast r_{12}+[r_{32},r_{12}]-r_{32}\ast r_{21}-[r_{31},r_{21}]+[r_{32},r_{13}]
\\=&-(\sigma_{23}+\sigma_{13})S(r)+r_{32}\ast(r_{13}-r_{31})
\\=&-(\sigma_{23}+\sigma_{13})S(r)+\sum_i(I\otimes I\otimes L_{\ast}(a_i)
-R_{\ast}(a_i)\otimes I\otimes I)(\tau\otimes I)(b_i\otimes (r-\tau(r))).\end{align*}
Using Eqs.~(\ref{Iv11})-(\ref{Iv12}) and (\ref{Op1}), we have for all $\zeta,\eta,\theta\in A^{*}$,
\begin{align*}&
\langle(I\otimes R_{\prec}(x)L_{\ast} (b_i)\otimes I+I\otimes R_{\prec}(x)\otimes \mathrm{ad} (b_i))(a_i\otimes(r-\tau(r))),
\zeta\otimes\eta\otimes \theta\rangle
\\=&\langle r-\tau(r),L_{\ast}^{*} (b_i)R_{\prec}^{*}(x)\eta\otimes \theta+R_{\prec}^{*}(x)\eta\otimes \mathrm{ad}^{*} (b_i)\theta\rangle
\langle a_i,\zeta\rangle
\\=&\langle T_{r-\tau(r)}(L_{\ast}^{*} (b_i)R_{\prec}^{*}(x)\eta)+\mathrm{ad} (b_i)T_{r-\tau(r)}(R_{\prec}^{*}(x)\eta),\theta \rangle
\langle a_i,\zeta\rangle
\\=&\langle-\mathrm{ad} (T_{r}(\zeta))T_{r-\tau(r)}(R_{\prec}^{*}(x)\eta)+\mathrm{ad} (T_{r}(\zeta))T_{r-\tau(r)}(R_{\prec}^{*}(x)\eta),\theta \rangle
\\=&0.\end{align*}
By Eqs.~(\ref{Iv11}), (\ref{Iv14}), (\ref{r70}) and (\ref{Op1}), we get for all $\zeta,\eta,\theta\in A^{*}$,
\begin{align*}&
\langle(L_{\ast}(x)\otimes I \otimes R_{\cdot}(a_i)+I\otimes I\otimes L_{\cdot}(x)L_{\ast} (a_i)-R_{\ast} (a_i)\otimes I\otimes L_{\cdot}(x))(\tau\otimes I)(b_i\otimes (\tau(r)-r)),
\zeta\otimes\eta\otimes\theta\rangle
\\=&\langle \tau(r)-r,L_{\ast}^{*} (x)\zeta\otimes R_{\cdot}^{*}(a_i) \theta+\zeta\otimes L_{\ast}^{*} (a_i)L_{\cdot}^{*}(x)\theta
-R_{\ast}^{*} (a_i)\zeta\otimes L_{\cdot}^{*}(x)\theta
\rangle
\langle b_i,\eta\rangle
\\=&-\langle R_{\cdot}(a_i)T_{r-\tau(r)}(L_{\ast}^{*} (x)\zeta)+L_{\cdot}(x)L_{\ast} (a_i)T_{r-\tau(r)}(\zeta)
-L_{\cdot}(x)T_{r-\tau(r)}(R_{\ast}^{*} (a_i)\zeta),\theta\rangle\langle b_i,\eta\rangle
\\=&\langle R_{\cdot}(a_i)\mathrm{ad}(x)T_{r-\tau(r)}(\zeta)+L_{\cdot}(x)L_{\ast} (a_i)T_{r-\tau(r)}(\zeta)
-L_{\cdot}(x)R_{\ast}(a_i)T_{r-\tau(r)}(\zeta),\theta\rangle\langle b_i,\eta\rangle
\\=&\langle R_{\cdot}(a_i)\mathrm{ad}(x)T_{r-\tau(r)}(\zeta)+L_{\cdot}(x)\mathrm{ad}(a_i)T_{r-\tau(r)}(\zeta)
,\theta\rangle\langle b_i,\eta\rangle
\\=&\langle \mathrm{ad}(x\cdot T_{\tau(r)}(\eta))T_{r-\tau(r)}(\zeta),\theta\rangle.\end{align*}
According to Eqs.~(\ref{Iv11}), (\ref{Iv15}), (\ref{r70}) and (\ref{Op1}), we get for all $\zeta,\eta,\theta\in A^{*}$,
\begin{align*}&
\langle(L_{\ast}( x\cdot a_i )\otimes I \otimes I)(\tau\otimes I)(b_i\otimes r)- 
(L_{\ast}( x\cdot b_i )\otimes I \otimes I)(\tau\otimes I)(a_i\otimes \tau(r)),
\zeta\otimes\eta\otimes\theta\rangle\\=&
\langle b_i,\eta\rangle\langle L_{\ast}^{*}( x\cdot a_i )\zeta\otimes \theta,r\rangle
-\langle a_i,\eta\rangle\langle L_{\ast}^{*}( x\cdot b_i )\zeta\otimes \theta,\tau(r)\rangle
\\=&
\langle T_{r}(L_{\ast}^{*}( x\cdot T_{\tau(r)}(\eta) )\zeta),\theta\rangle
-\langle T_{\tau(r)}( L_{\ast}^{*}( x\cdot T_{r}(\eta) )\zeta), \theta\rangle
\\=&\langle T_{r-\tau(r)}(L_{\ast}^{*}( x\cdot T_{\tau(r)}(\eta) )\zeta),\theta\rangle
-\langle T_{\tau(r)}( L_{\ast}^{*}( x\cdot T_{r-\tau(r)}(\eta) )\zeta), \theta\rangle
\\=&-\langle \mathrm{ad}( x\cdot T_{\tau(r)}(\eta) T_{r-\tau(r)}(\zeta),\theta\rangle
-\langle T_{\tau(r)}( L_{\ast}^{*}( T_{r-\tau(r)}(R_{\prec}^{*}(x)\eta) )\zeta), \theta\rangle
\\=&-\langle \mathrm{ad}( x\cdot T_{\tau(r)}(\eta) T_{r-\tau(r)}(\zeta),\theta\rangle
+\langle  L_{\ast}^{*}(T_{r-\tau(r)}(\zeta))R_{\prec}^{*}(x)\eta, T_{r}(\theta)\rangle
\\=&-\langle \mathrm{ad}( x\cdot T_{\tau(r)}(\eta) T_{r-\tau(r)}(\zeta),\theta\rangle
+\langle \eta, R_{\prec}(x) L_{\ast}(T_{r-\tau(r)}(\zeta))T_{r}(\theta)\rangle.\end{align*}
In view of Eqs.~(\ref{Iv11}), (\ref{Iv13}), (\ref{r8}) and (\ref{Op1}), 
we obatin for all $\zeta,\eta,\theta\in A^{*}$,
\begin{align*}&
\langle(L_{\ast}(b_i\prec x)\otimes I \otimes I +I \otimes L_{\ast}(b_i\prec x)\otimes I
+L_{\ast}(x)R_{\succ}(b_i)\otimes I \otimes I)((\tau(r)-r)\otimes a_i)
,\zeta\otimes\eta\otimes \theta\rangle\\=& \langle\tau(r)-r,L_{\ast}^{*}(b_i\prec x)\zeta\otimes \eta+\zeta \otimes L_{\ast}^{*}(b_i\prec x)\eta
+R_{\succ}^{*}(b_i)L_{\ast}^{*}(x)\zeta\otimes \eta\rangle
\langle a_i,\theta\rangle
\\=&-\langle T_{r-\tau(r)}(L_{\ast}^{*}(b_i\prec x)\zeta)+L_{\ast}(b_i\prec x)T_{r-\tau(r)}(\zeta)
+T_{r-\tau(r)}(R_{\succ}^{*}(b_i)L_{\ast}^{*}(x)\zeta), \eta\rangle
\langle a_i,\theta\rangle
\\=&\langle (\mathrm{ad}(b_i\prec x)-L_{\ast}(b_i\prec x)-L_{\prec}(b_i)\mathrm{ad}(x))T_{r-\tau(r)}(\zeta), \eta\rangle
\langle a_i,\theta\rangle
\\=&\langle -R_{\prec}(x)R_{\ast}(T_{r}(\theta))T_{r-\tau(r)}(\zeta), \eta\rangle.
\end{align*}
Combining the above equalities, we get that Eq.~(\ref{CA1}) holds. By the same token, Eqs.~(\ref{CA2})-(\ref{CA8}) hold. 
Combining Theorem \ref{Bc1}, we complete the proof.
\end{proof}

\begin{ex} Let $(A=ke_1\oplus ke_2,\succ,\prec,\ast)$ be the 2-dimensional
 coherent noncommutative 
pre-Poisson algebra given in Example \ref{Ae}. It is easy to check that
$r=e_1\otimes e_2$ is a non-symmetric solution of the NPP-YBE in $(A,\succ,\prec,\ast)$ 
and $r-\tau(r)=e_1\otimes e_2-e_2\otimes e_1$ is invariant.
\end{ex}

Let $(A, \succ,\prec,\ast,\Delta_{\succ,r},\Delta_{\prec,r},\delta_r)$ be a noncommutative pre-Poisson bialgebra and $r\in A\otimes A$, 
where the comultiplications
$\Delta_{\succ,r},\Delta_{\prec,r},\delta_r$ are defined by Eqs.~(\ref{CB1})-(\ref{CB2}) and (\ref{CB3}). According to 
Remark \ref{Bc}, there is a coherent noncommutative pre-Poisson algebra structure $(\succ_r,\prec_r,\ast_r)$ on $A^{*}$.
By direct computations, \begin{small}
\begin{align}&\label{Ic1}\zeta\succ_{r}\eta=R_{\cdot}^{*}(T_{\tau(r)}(\zeta))\eta-L_{\prec}^{*}(T_{r}(\eta))\zeta,
\\&\label{Ic2}\zeta\prec_{r}\eta=-R_{\succ}^{*}(T_{r}(\zeta))\eta+L_{\cdot}^{*} (T_{\tau(r)}(\eta))\zeta,\\&
\label{Ic3}\zeta\ast_{r}\eta=\mathrm{ad}^{*} (T_{r}(\zeta))\eta-R_{\ast}^{*}(T_{\tau(r)}(\eta))\zeta, 
\end{align}
\end{small}
for all $\zeta,\eta\in A^{*}$. And the sub-adjacent noncommutative Poisson 
algebra structure $(\cdot_r,[ \ , \ ]_{r})$ on $A^{*}$ is
given by 
\begin{equation*}\zeta\cdot_{r}\eta=\zeta\succ_{r}\eta+\zeta\prec_{r}\eta
, \ \ \ [ \zeta , \eta ]_{r}=\zeta\ast_{r}\eta-\eta\ast_{r}\zeta,
\end{equation*}
where $\mathrm{ad}=L_{\ast}-R_{\ast},~ \cdot=\succ+\prec$.

\begin{defi} \label{Qt1}
 Let $(A,\succ,\prec,\ast)$ be a coherent noncommutative pre-Poisson algebra 
 and $r\in A\otimes A$. If $r$ is a solution of the NPP-YBE in $(A,\succ,\prec,\ast)$
 and $r-\tau(r)$ is invariant, then the noncommutative pre-Poisson 
  bialgebra $(A, \succ,\prec,\ast,\Delta_{\succ,r},\Delta_{\prec,r},\delta_{r})$ induced by $r$
   is called a {\bf quasi-triangular} noncommutative pre-Poisson
 bialgebra. In particular, $(A, \succ,\prec,\ast,\Delta_{\succ,r},\Delta_{\prec,r},\delta_{r})$
is called a {\bf triangular} noncommutative pre-Poisson bialgebra if $r$ is symmetric, 
where $\Delta_{\succ,r},\Delta_{\prec,r}$ and $\delta_{r}$ 
are given by Eqs.~(\ref{CB1})-(\ref{CB2}) and (\ref{CB3}) respectively.
\end{defi}

Combining Theorem 2.18 \cite {Wa} and Theorem 2.14 \cite {Wbl}, we get the following results:
\begin{thm}\label{QB0} Let $(A,\succ,\prec,\ast)$ be a coherent noncommutative pre-Poisson algebra and
$r\in A\otimes A$. Suppose that $r-\tau(r)$ is invariant. Then the following conditions are equivalent:
 \begin{enumerate}
\item $r$ is a solution of the NPP-YBE in $(A,\succ,\prec,\ast)$.
 \item $(A^{*},\cdot_r, [ \ , \ ]_r)$ is a coherent Poisson algebra and the linear maps
$T_{r},T_{\tau(r)}$ are both Poisson algebra
 homomorphisms from $(A^{*},\cdot_r, [ \ , \ ]_r)$ to $(A,\cdot, [ \ , \ ])$.
 \item $(A^{*},\succ_r,\prec_r,\ast_r)$ is a coherent noncommutative pre-Poisson algebra and the linear maps
$T_{r},T_{\tau(r)}$ are noncommutative pre-Poisson algebra
 homomorphisms from $(A^{*},\succ_r,\prec_r,\ast_r)$ to $(A,\succ,\prec,\ast)$.
 \end{enumerate}
\end{thm}

\begin{cor} \label{QB} Let $(A, \succ,\prec,\ast,\Delta_{\succ,r},\Delta_{\prec,r},\delta_{r})$ be 
a quasi-triangular noncommutative pre-Poisson bialgebra.
 Then  \begin{enumerate}
\item $T_{r},T_{\tau(r)}$ are both Poisson algebra
 homomorphisms from $(A^{*},\cdot_r, [ \ , \ ]_r)$ to $(A,\cdot, [ \ , \ ])$.
 \item $T_{r},T_{\tau(r)}$ are noncommutative pre-Poisson algebra
 homomorphisms from $(A^{*},\succ_r,\prec_r,\ast_r)$ to $(A,\succ,\prec,\ast)$.
 \end{enumerate}
\end{cor}

\subsection{Factorizable noncommutative pre-Poisson bialgebras} 
We begin by introducing the notion of a factorizable noncommutative pre-Poisson bialgebra,
 which constitutes a key subclass of the quasi-triangular family.
  We then prove that the double of any noncommutative pre-Poisson bialgebra 
  naturally carries the structure of such a factorizable bialgebra.

\begin{defi} \label{Qt} A quasi-triangular noncommutative pre-Poisson bialgebra
 $(A, \succ,\prec,\ast,\Delta_{\succ,r},\Delta_{\prec,r},\delta_{r})$ is
called factorizable if  
 the linear map $T_{r-\tau(r)}:A^{*}\longrightarrow A$ given by Eq.~(\ref{Op1}) is a linear isomorphism of vector spaces, 
where $\Delta_{\succ,r}, \Delta_{\prec,r},\delta_{r} $
are defined by Eqs.~(\ref{CB1})-(\ref{CB2}) and (\ref{CB3}) respectively.\end{defi}

\begin{pro} \label{Qf} Let $(A, \succ,\prec,\ast,\Delta_{\succ,r},\Delta_{\prec,r},\delta_{r})$ be a quasi-triangular
 (factorizable) noncommutative pre-Poisson bialgebra. 
 Then $(A, \succ,\prec,\ast,\Delta_{\succ,\tau(r)},\Delta_{\prec,\tau(r)},\delta_{\tau(r)})$
 is also a quasi-triangular
 (factorizable) noncommutative pre-Poisson bialgebra.\end{pro}
\begin{proof}
Since $(A, \succ,\prec,\ast,\Delta_{\succ,r},\Delta_{\prec,r},\delta_{r})$ is a quasi-triangular
 (factorizable) noncommutative pre-Poisson bialgebra, 
 $D(r)=S(r)=0$ and $r-\tau(r)$ is invariant.
Note that
\begin{align*}&D(\tau(r))=r_{21}\cdot r_{31}-r_{32}\succ r_{21}-r_{31}\prec r_{32}\\=&
D(r)+(r_{21}-r_{12})\cdot r_{31}+r_{12}\cdot(r_{31}-r_{13})-r_{32}\succ(r_{21}-r_{12})\\&+(r_{23}-r_{32})\succ r_{12}
+(r_{13}-r_{31})\prec r_{32}+r_{13}\prec(r_{23}-r_{32})
\\=&D(r)-\sum_i(R_{\cdot}(b_i)\otimes I-I\otimes L_{\succ}(b_i))(r-\tau(r))\otimes a_i
+a_i\otimes( R_{\succ}(b_i)\otimes I+I \otimes L_{\prec}(b_i))(r-\tau(r))
\\&-(L_{\cdot}(a_i)\otimes I\otimes I-I\otimes I\otimes R_{\prec}(a_i))(\tau\otimes I)(b_i\otimes(r-\tau(r))
\\=&0.\end{align*}
An analogous argument shows that $S(\tau(r))=0$. The proof is completed.
\end{proof}

\begin{pro}
Let $(A, \succ,\prec,\ast,\Delta_{\succ,r},\Delta_{\prec,r},\delta_{r})$ be a factorizable noncommutative pre-Poisson bialgebra.
Then $\mathrm{Im}(T_{r}\oplus T_{\tau(r)})$ is a noncommutative pre-Poisson subalgebra of the direct sum noncommutative pre-Poisson
algebra $A\oplus A$,
which is isomorphic to the noncommutative pre-Poisson algebra $(A^{*},\succ_r,\prec_r,\ast_r)$. 
Furthermore, any $x\in A$ has a unique
decomposition $x=x_1-x_2$, where $(x_1,x_2)\in \mathrm{Im}(T_{r}\oplus T_{\tau(r)})$ and
\begin{equation*}T_{r}\oplus T_{\tau(r)}:A^{*}\longrightarrow A\oplus A,~~(T_{r}\oplus T_{\tau(r)})(\zeta)=(T_{r}(\zeta),
 T_{\tau(r)}(\zeta)),~\forall~\zeta\in A^{*}. \end{equation*}
\end{pro}

\begin{proof} In view of Definition \ref{Qt} and Corollary \ref{QB}, $T_{r}, T_{\tau(r)}$ are noncommutative pre-Poisson algebra 
homomorphisms and $T_{r-\tau(r)}$ is a linear isomorphism. It follows that
$T_{r}\oplus T_{\tau(r)}$ is a noncommutative pre-Poisson algebra homomorphism and $\mathrm{Ker}(T_{r}\oplus T_{\tau(r)})=0$.
Therefore, $\mathrm{Im}(T_{r}\oplus T_{\tau(r)})$ is isomorphic to $(A^{*},\succ_r,\prec_r,\ast_r)$ as noncommutative pre-Poisson algebras.
For all $x\in A$, 
\begin{equation*}x=T_{r-\tau(r)}T_{r-\tau(r)}^{-1}(x)=T_{r}T_{r-\tau(r)}^{-1}(x)-T_{\tau(r)}T_{r-\tau(r)}^{-1}(x)=x_1-x_2
,\end{equation*}
where $x_1=T_{r}T_{r-\tau(r)}^{-1}(x)\in \mathrm{Im}(T_{r}),~
x_2=T_{\tau(r)}T_{r-\tau(r)}^{-1}(x)\in \mathrm{Im}(T_{\tau(r)})$.
The proof is finished.
\end{proof}

\begin{pro}
 Let $(A, \succ_{A},\prec_{A},\ast_{A}, \Delta_{\succ,r},  \Delta_{\prec,r},\delta_{r})$
 be a factorizable noncommutative pre-Poisson bialgebra. Then the double noncommutative pre-Poisson algebra $(D=A\oplus A^{*},\succ_D,\prec_D,\ast_D)$
is isomorphic to the direct sum $A\oplus A$ of noncommutative pre-Poisson algebras, 
where $\succ_D,\prec_D$ and $\ast_D$ are given by Eqs.~(\ref{Db1})-(\ref{Db3}) respectively.
\end{pro}

\begin{proof} Since $(A, \succ_{A},\prec_{A},\ast_{A}, \Delta_{\succ,r},  \Delta_{\prec,r},\delta_{r})$
 is a factorizable noncommutative pre-Poisson bialgebra,
$T_{r-\tau(r)}$ is a linear isomorphism and $r-\tau(r)$ is invariant.
Define $\varphi:A\oplus A^{*}\longrightarrow A\oplus A$ by
\begin{equation}\label{FD1}\varphi(x,\zeta)=(x+T_{r}(\zeta),x+T_{\tau(r)}(\zeta)),~\forall~(x,\zeta)\in A\oplus A^{*}.\end{equation}
Then $\varphi$ is a bijection.
By Eq.~(\ref{Ic2}), we get
\begin{small}
\begin{align*}&\langle\eta,-L_{\prec_{A^*}}^{*}(\zeta)x+T_{r}(R_{\cdot_A}^{*}(x)\zeta\rangle
\\=&\langle -\zeta\prec \eta ,x\rangle+\langle \eta,T_{r}(R_{\cdot_A}^{*}(x)\zeta\rangle
\\=&\langle R_{\succ}^{*}(T_{r}(\zeta))\eta-L_{\cdot}^{*} (T_{\tau(r)}(\eta))\zeta,x\rangle+\langle \eta,T_{r}(R_{\cdot_A}^{*}(x)\zeta\rangle
\\=&\langle x\succ_A T_{r}(\zeta),\eta\rangle-\langle \zeta,T_{\tau(r)}(\eta)\cdot_A x\rangle+\langle \eta,T_{r}(R_{\cdot_A}^{*}(x)\zeta\rangle
\\=&\langle x\succ_A T_{r}(\zeta),\eta\rangle-\langle T_{r}(R_{\cdot_A}^{*}(x)\zeta),\eta\rangle+\langle \eta,T_{r}(R_{\cdot_A}^{*}(x)\zeta\rangle
\\=&\langle x\succ_A T_{r}(\zeta),\eta\rangle,
\end{align*}\end{small}
which implies that
\begin{equation}\label{FD2}x\succ_{A} T_{r}(\zeta)=-L_{\prec_{A^*}}^{*}(\zeta)x+T_{r}(R_{\cdot_A}^{*}(x)\zeta.\end{equation}
Similarly, using Eqs.~(\ref{Ic2}) and (\ref{Iv11}), we have
\begin{equation}\label{FD3}x\succ_{A} T_{\tau(r)}(\zeta)=-L_{\prec_{A^*}}^{*}(\zeta)x+T_{\tau(r)}(R_{\cdot_A}^{*}(x)\zeta.\end{equation}
By Eqs.~(\ref{Db2}) and (\ref{FD2})-(\ref{FD3}), we get 
\begin{small}
\begin{align*}&\varphi(x\succ_D \zeta)=\varphi(-L_{\prec_{A^*}}^{*}(\zeta)x+R_{\cdot_A}^{*}(x)\zeta)
\\
=&(-L_{\prec_{A^*}}^{*}(\zeta)x+T_{r}(R_{\cdot_A}^{*}(x)\zeta),-L_{\prec_{A^*}}^{*}(\zeta)x+T_{\tau(r)}(R_{\cdot_A}^{*}(x)\zeta))
\\=&(x\succ_A T_{r}(\zeta),x\succ_A T_{\tau(r)}(\zeta))
=(x,x)\succ (T_{r}(\zeta),T_{\tau(r)}(\zeta))
\\=&\varphi(x)\succ_A \varphi(\zeta).
\end{align*}\end{small}
Analogously, $\varphi(\zeta\succ_D x)=\varphi(\zeta)\succ_A \varphi(x) $. Thus,
$\varphi((x,\zeta)\succ_D (y,\eta))=\varphi(x,\zeta)\succ \varphi(y,\eta) $ for all
$(x,\zeta), (y,\eta)\in A\oplus A^{*}$. Analogously, $\varphi((x,\zeta)\prec_D (y,\eta))=\varphi(x,\zeta)\prec \varphi(y,\eta) $
and
 $\varphi((x,\zeta)\ast_D (y,\eta))=\varphi(x,\zeta)\ast \varphi(y,\eta) $.
In all, $\varphi$ is an isomorphism of noncommutative pre-Poisson algebras.
The proof is completed.
\end{proof}

\begin{thm} \label{Ft}
Let $(A, \succ,\prec,\ast, \Delta_{\succ,r},  \Delta_{\prec,r},\delta_{r})$ be a noncommutative pre-Poisson
 bialgebra. Assume that $\{e_1,...,e_n\}$ is a basis of $A$ and $\{e^{*}_1,...,e^{*}_n\}$ is the dual basis.
 Let \begin{equation*}r=\sum_{i=1}^{n}e_{i}\otimes e_{i}^{*}\in A\otimes A^{*}\subseteq D\otimes D.\end{equation*}
Then $(D,\succ_D,\prec_D,\ast_D,\Delta_{\succ,r},  \Delta_{\prec,r},\delta_{r})$ 
 is a factorizable noncommutative pre-Poisson bialgebra with 
$\Delta_{\succ,r},  \Delta_{\prec,r},\delta_{r}$ given by Eqs.~(\ref{CB1})-(\ref{CB2}) and (\ref{CB3}) respectively.
\end{thm}

\begin{proof} On the basis of Theorem 2.21 \cite{Wa} and Theorem 2.18 \cite{Wbl}, we get that
$r-\tau(r)$ is invariant and $S(r)=D(r)=0$. Thus, 
$(A, \succ,\prec,\ast, \Delta_{\succ,r},  \Delta_{\prec,r},\delta_{r})$
is a quasi-triangular noncommutative pre-Poisson bialgebra.
 Furthermore, the linear maps $T_{r},T_{\tau(r)}:D^{*}\longrightarrow D$ are respectively defined by 
$T_{r}(\zeta,x)=\zeta,~T_{\tau(r)}(\zeta,x)=x$ for all $x\in A,\zeta\in A^{*}$. Thus,
$T_{r-\tau(r)}(\zeta,x)=(\zeta,-x)$ is a linear isomorphism. Hence, 
$(D,\succ_D,\prec_D,\ast_D,\Delta_{\succ,r},\Delta_{\prec,r},\delta_{r})$
is a factorizable noncommutative pre-Poisson bialgebra.
\end{proof}

\section{Relative Rota-Baxter operators and quadratic Rota-Baxter noncommutative pre-Poisson algebras }
\subsection{Relative Rota-Baxter operators and the NPP-YBE}
This section focuses on solutions of the NPP-YBE with invariant skew-symmetric parts, 
studying them specifically via relative Rota-Baxter operators of weight $\lambda$ on noncommutative pre-Poisson algebras

 \begin{pro}
Let $(A,\succ,\prec,\ast)$ be a coherent noncommutative pre-Poisson algebra and $r\in A\otimes A$ be skew-symmetric and invariant. Define the
operations $\succ_r,\prec_r,\ast_r:A^{*}\otimes A^{*}\longrightarrow A^{*}$ by
\begin{align}&\label{Apa1}\zeta\succ_{r}\eta=R_{\cdot}^{*}(T_{r}(\zeta))\eta=-L_{\prec}^{*}(T_{r}(\eta))\zeta,
\\&\label{Apa2}\zeta\prec_{r}\eta=-R_{\succ}^{*}(T_{r}(\zeta))\eta=L_{\cdot}^{*} (T_{r}(\eta))\zeta,\\&
\label{Apa3}\zeta\ast_{r}\eta=-\mathrm{ad}^{*} (T_{r}(\zeta))\eta=R_{\ast}^{*}(T_{r}(\eta))\zeta,~~\forall~\zeta,\eta\in A^{*}.
\end{align}
Then 
$(A^{*},\succ_r,\prec_r,\ast_r,R_{\cdot}^{*},-L_{\prec}^{*},-R_{\succ}^{*},L_{\cdot}^{*},R^{*}_{\ast}-L_{\ast}^{*},R^{*}_{\ast})$
is an A-noncommutative pre-Poisson algebra
and $(A^{*},\cdot_r,[ \ , \ ]_{r},R_{\prec}^{*}, L_{\succ}^{*},-L_{\ast}^{*})$ is an A-noncommutative Poisson algebra, where
the Poisson algebra structure $(\cdot_r,[ \ , \ ]_{r})$ on $A^{*}$ is
given by 
\begin{small}\begin{align}&\label{Apa4}\zeta\cdot_{r}\eta=R_{\cdot}^{*}(T_{r}(\zeta))\eta-R_{\succ}^{*}(T_{r}(\zeta))\eta
=-L_{\prec}^{*}(T_{r}(\eta))\zeta+L_{\cdot}^{*} (T_{r}(\eta))\zeta
,\\&\label{Apa5}
[ \zeta , \eta ]_{r}=\mathrm{ad}^{*} (T_{r}(\eta))\zeta-\mathrm{ad}^{*} (T_{r}(\zeta))\eta
=R_{\ast}^{*}(T_{r}(\zeta))\eta-R_{\ast}^{*}(T_{r}(\eta))\zeta.
\end{align}\end{small}
\end{pro}

\begin{proof} 
By Eqs.~(\ref{Apa1}), (\ref{Apa3}), (\ref{Apa5}), (\ref{Iv10}), (\ref{Iv13}) and (\ref{Np4}), we obtain
\begin{align*}
	&\zeta \ast_{r} L_{\prec}^{*}(x)\eta+\eta \succ_{r}(R_{\ast}^{*}(x)\zeta)-L_{\prec}^{*}(x)[\zeta, \eta]_{r}\\=&
-\mathrm{ad}^{*} (T_{r}(\zeta))L_{\prec}^{*}(x)\eta+R_{\cdot}^{*}(T_{r}(\eta))R_{\ast}^{*}(x)\zeta
+L_{\prec}^{*}(x)(\mathrm{ad}^{*} (T_{r}(\zeta))\eta-\mathrm{ad}^{*} (T_{r}(\eta))\zeta)
\\=&R_{\ast}^{*}(T_{r}(L_{\prec}^{*}(x)\eta))\zeta+R_{\cdot}^{*}(T_{r}(\eta))R_{\ast}^{*}(x)\zeta
-L_{\prec}^{*}(x)(R_{\ast}^{*} (T_{r}(\eta))\zeta+\mathrm{ad}^{*} (T_{r}(\eta))\zeta)
\\=&-R_{\ast}^{*}(R_{\succ}(x)T_{r}(\eta))\zeta+R_{\cdot}^{*}(T_{r}(\eta))R_{\ast}^{*}(x)\zeta
-L_{\prec}^{*}(x)L_{\ast}^{*} (T_{r}(\eta))\zeta
\\=&0,\end{align*}
which means that Eq.~(\ref{ppmp eq1.22}) holds in the case that 
$l_{\ast}=R_{\ast}^{*}-L_{\ast}^{*},~ r_{\ast}=R_{\ast}^{*},~ l_{\succ}=R_{\cdot}^{*},~r_{\succ}=-L_{\prec}^{*},~
 l_{\prec}=-R_{\succ}^{*},~r_{\prec}=L_{\cdot}^{*}$.
 Similarly, we prove that Eqs.~(\ref{ppmp eq1.22})-(\ref{ppmp eq1.30}) hold
and $(A^{*},\ast_r,R^{*}_{\ast}-L_{\ast}^{*},R^{*}_{\ast})$
is an A-pre-Lie algebra. By Lemma 3.4 \cite{Wa}, $(A^{*},\succ_r,\prec_r,R_{\cdot}^{*},-L_{\prec}^{*},-R_{\succ}^{*},L_{\cdot}^{*})$
is an A-dendriform algebra. 
Thus, $(A^{*},\succ_r,\prec_r,\ast_r,R_{\cdot}^{*},
-L_{\prec}^{*},-R_{\succ}^{*},L_{\cdot}^{*},R^{*}_{\ast}-L_{\ast}^{*},R^{*}_{\ast})$
is an A-noncommutative pre-Poisson algebra. 
Then $(A^{*},\cdot_r,[ \ , \ ]_{r},R_{\prec}^{*}, L_{\succ}^{*},-L_{\ast}^{*})$ is 
an A-noncommutative Poisson algebra holds naturally.
This completes the proof.
\end{proof}

\begin{pro} \label{Ya1} Let $(A,\succ,\prec,\ast)$ be a noncommutative pre-Poisson algebra and
$r\in A\otimes A$. Then 
 \begin{enumerate}
 \item $r$ is a solution of the
equation $ D_2(r)=r_{13}\cdot r_{23}-r_{12}\succ r_{13}-r_{23}\prec r_{12}=0$
in $(A,\succ,\prec)$
 if and only if the following equation holds:
 \begin{equation*}T_{r}(\zeta)\cdot T_{r}(\eta)=T_{r}(R_{\prec}^{*}(T_{r}(\zeta))\eta
 +L_{\succ}^{*}(T_{\tau(r)}(\eta))\zeta),~~\forall~\zeta,\eta\in A^{*}.
\end{equation*}
\item $r$ is a solution of the
D-equation 
$ D(r)=r_{12}\cdot r_{13}-r_{23}\succ r_{12}-r_{13}\prec r_{23}=0$ in $(A,\succ,\prec)$
 if and only if the following equation holds:
 \begin{equation*}T_{r}(\zeta)\prec T_{r}(\eta)=T_{r}(L_{\cdot}^{*}(T_{\tau(r)}(\eta))\zeta
 -R_{\succ}^{*}(T_{r}(\zeta))\eta),~~\forall~\zeta,\eta\in A^{*}.
\end{equation*}
\item $r$ is a solution of the equation 
$D_3(r)=r_{23}\cdot r_{12}-r_{13}\succ r_{23}-r_{12}\prec r_{13}=0$
if and only if the following equation holds:
\begin{equation*}T_{r}(\zeta)\succ T_{r}(\eta)=T_{r}(R_{\cdot}^{*}(T_{r}(\zeta))\eta
 -L_{\prec}^{*}(T_{\tau(r)}(\eta))\zeta),~~\forall~\zeta,\eta\in A^{*}.
 \end{equation*}
 \item  $r$ is a solution of the S-equation
$S(r)=r_{12}\ast r_{23}-r_{12}\ast r_{13}+[r_{13}, r_{23}]=0$
if and only if the following equation holds:
  \begin{equation*}[T_{r}(\zeta), T_{r}(\eta)]=T_{r}(L_{\ast}^{*}(T_{\tau(r)}(\eta))\zeta-
 L_{\ast}^{*}(T_{r}(\zeta))\eta),~~\forall~\zeta,\eta\in A^{*}.
\end{equation*}
\item $r$ is a solution of the equation 
$S_{1}(r)= r_{23}\ast r_{12}-r_{23}\ast r_{13}-[r_{12}, r_{13}]=0$
if and only if the following equation holds:
  \begin{equation*}T_{r}(\zeta)\ast T_{r}(\eta)=T_{r}(R_{\ast}^{*}(T_{r}(\eta))\zeta-\mathrm{ad}^{*}(T_{\tau(r)}(\zeta))\eta
 ),~~\forall~\zeta,\eta\in A^{*}.
\end{equation*}
\end{enumerate}
\end{pro}

\begin{proof} According to Eq.~\eqref{Op1}, we have for all $\zeta,\eta,\theta\in A^{*}$, 
\begin{align*}\langle \zeta\otimes \eta\otimes \theta,r_{13}\cdot r_{23}\rangle
&=\sum_{i,j}\langle \zeta\otimes \eta\otimes \theta,a_i\otimes a_j \otimes b_i\cdot  b_j\rangle
\\&=\sum_{i,j}\langle \theta,b_i\cdot b_j\rangle\langle \zeta,a_i\rangle \langle \eta,a_j\rangle 
=\langle \theta,T_{r}(\zeta)\cdot T_{r}(\eta) \rangle
,\end{align*}
\begin{align*}\langle \zeta\otimes \eta\otimes \theta,r_{23}\prec r_{12}\rangle
&=\sum_{i,j}\langle \zeta\otimes \eta\otimes \theta,a_j \otimes a_i\prec b_j\otimes b_i\rangle
\\&=\sum_{i,j}\langle \zeta,a_j\rangle\langle \theta,b_i\rangle \langle \eta,b_i\prec b_j\rangle 
=\langle  \eta,T_{r}(\zeta) \prec T_{\tau(r)}(\theta)\rangle
\\&=\langle  L_{\prec}^{*}(T_{r}(\zeta))\eta,T_{\tau(r)}(\theta)\ \rangle
=\langle \theta, T_{r}(L_{\prec}^{*}(T_{r}(\zeta))\eta) \rangle,\end{align*}
 \begin{align*}\langle \zeta\otimes \eta\otimes \theta,r_{12}\succ r_{13}
  \rangle
 &=\sum_{i,j}\langle \zeta\otimes \eta\otimes \theta,a_i \succ a_j\otimes b_i\otimes b_j\rangle
 \\& =\sum_{i,j}\langle \eta,b_i\rangle\langle \zeta,a_i\succ a_j\rangle \langle \theta,b_j\rangle 
 =\langle \zeta,T_{\tau(r)}(\eta)\succ T_{\tau(r)}(\theta) \rangle
  \\&=\langle L_{\succ}^{*}(T_{\tau(r)}(\eta))\zeta,T_{\tau(r)}(\theta)\rangle
 =\langle \theta,T_{r}(L_{\succ}^{*}(T_{\tau(r)}(\eta))\zeta\rangle.\end{align*}
 Thus, Item (a) holds.
Similarly, Items (b)-(e) hold.\end{proof}

\begin{pro}\label{Ya2} Let $(A,\succ,\prec,\ast)$ be a noncommutative pre-Poisson algebra and
$r\in A\otimes A$. Assume that $r-\tau(r)$ is invariant. Then 
\begin{enumerate}
 \item $r$ is a solution of the
S-equation $S(r)=0$ if and only if 
$r$ is a solution of $S_1(r)=0$.
\item $r$ is a solution of $D(r)=0$ if and only if 
$r$ is a solution of $D_3(r)=D_2(r)=0$.
\end{enumerate}
\end{pro} 

\begin{proof}
Note that $D_3(r)+D_2(r)=-D(r)$ and
\begin{align*}S_{1}(r)=&-\sigma_{123}S(r)+r_{23}\ast (r_{31}-r_{13})
+[r_{21},r_{31}-r_{13}]+r_{23}\ast (r_{12}-r_{21})+[r_{21}-r_{12},r_{13}]
\\=&-\sigma_{123}S(r)-\sum_{i}(I\otimes I\otimes L_{\ast}(b_i)
+\mathrm{ad}(b_i)\otimes I\otimes I)(\tau\otimes I)(a_i\otimes (r-\tau(r)))
\\&+(I\otimes L_{\ast}(a_i)+\mathrm{ad}(a_i)\otimes I) (r-\tau(r))\otimes b_i,\end{align*} 
\begin{align*} D_3(r)=&\sigma_{123}D(r)+r_{23}\cdot (r_{12}-r_{21})-(r_{12}-r_{21})\prec r_{13}
-(r_{13}-r_{31})\succ r_{23}-r_{21}\prec(r_{13}-r_{31})
\\=&\sigma_{123}D(r)+\sum_{i}(I\otimes L_{\cdot}(a_i)-R_{\prec}(a_i)\otimes I)(r-\tau(r))\otimes b_i\\&-(I\otimes I\otimes R_{\succ}(b_i)
+L_{\prec}(b_i)\otimes I\otimes I)
(\tau\otimes I)(a_i\otimes (r-\tau(r))),
\end{align*} 
\begin{align*} D_2(r)=&\sigma_{132}D(r)+r_{13}\cdot (r_{23}-r_{32})-(r_{23}-r_{32})\prec r_{12}+(r_{13}-r_{31})\cdot r_{32}+
r_{12}\succ (r_{31}-r_{13})
\\=&\sigma_{132}D(r)+\sum_{i}a_i\otimes
(I\otimes L_{\cdot}(b_i)-R_{\prec}(b_i)\otimes I)(r-\tau(r))\\&+(I\otimes I\otimes R_{\cdot}(a_i)
-L_{\succ}(a_i)\otimes I\otimes I)(\tau\otimes I)(b_i\otimes (r-\tau(r))).
\end{align*} 
Thus, the conclusions hold.
\end{proof}

The following Theorem is a generalization of Theorem \ref{Opp}.
\begin{thm} \label{Aa}
Let $(A,\succ,\prec,\ast)$ be a coherent noncommutative pre-Poisson algebra and $r\in A\otimes A$. Assume that $r-\tau(r)$ is invariant.
Then the following
conditions are equivalent.
 \begin{enumerate}
\item $r$ is a solution of the NPP-YBE in $(A,\succ,\prec,\ast)$
 such that $(A,\succ,\prec,\ast,\Delta_{\succ,r},\Delta_{\prec,r},\delta_{r})$ is a quasi-triangular noncommutative pre-Poisson bialgebra
  with $\Delta_{\succ,r},\Delta_{\prec,r},\delta_{r}$ given by Eqs.~(\ref{CB1})-(\ref{CB2}) and (\ref{CB3}) respectively.
\item $T_r$ is a relative Rota-Baxter operator of weight $-1$ on  $(A,\succ,\prec,\ast)$
 with respect to the A-noncommutative pre-Poisson algebra 
 $(A^{*},\succ_{r-\tau(r)},\prec_{r-\tau(r)},\ast_{r-\tau(r)},R_{\cdot}^{*},-L_{\prec}^{*},-R_{\succ}^{*},
 L_{\cdot}^{*},R^{*}_{\ast}-L_{\ast}^{*},R^{*}_{\ast})$,
 that is,
 \begin{small}
 \begin{align}\label{AD5}&
T_{r}(\zeta)\succ T_{r}(\eta)=T_{r}(R_{\cdot}^{*}(T_{r}(\zeta))\eta-L_{\prec}^{*}(T_{r}(\eta))\zeta
-\zeta\succ_{r-\tau(r)}\eta),\\&
\label{AD6}T_{r}(\zeta)\prec T_{r}(\eta)=T_{r}(-R_{\succ}^{*}(T_{r}(\zeta))\eta+L_{\cdot}^{*}(T_{r}(\eta))\zeta
-\zeta\prec_{r-\tau(r)}\eta),
\\&\label{AD7}T_{r}(\zeta)\ast T_{r}(\eta)=T_{r}((R^{*}_{\ast}-L_{\ast}^{*})(T_{r}(\zeta))\eta
 +R_{\ast}^*(T_{r}(\eta))\zeta-\zeta\ast_{r-\tau(r)}\eta).
 \end{align}\end{small}
\item $T_r$ is a relative Rota-Baxter operator of weight $-1$ on $(A,\cdot, [ \ , \ ])$
 with respect to the A-noncommutative Poisson algebra $(A^{*},\cdot_{r-\tau(r)},[ \ , \ ]_{r-\tau(r)},R_{\prec}^{*}, L_{\succ}^{*},-L_{\ast}^{*})$, that is,
 \begin{small}
 \begin{align}\label{AD8} &
 T_{r}(\zeta)\cdot T_{r}(\eta)=T_{r}(R_{\prec}^{*}(T_{r}(\zeta))\eta+ L_{\succ}^{*}(T_{r}(\eta))\zeta-\zeta\cdot_{r-\tau(r)}\eta),\\&
 \label{AD9} [T_{r}(\zeta),T_{r}(\eta)]=T_{r}(L_{\ast}^{*}(T_{r}(\eta))\zeta-L_{\ast}^{*}(T_{r}(\zeta))\eta-[\zeta,\eta]_{r-\tau(r)}), 
 \end{align}\end{small}
  \end{enumerate}
for all $\zeta,\eta\in A^{*}$, where $\succ_{r-\tau(r)},\prec_{r-\tau(r)},\ast_{r-\tau(r)}$ and $\cdot_{r-\tau(r)},
[ \ , \ ]_{r-\tau(r)}$ are given by Eqs.~(\ref{Apa1})-(\ref{Apa5}).
\end{thm}
  \begin{proof}
  Based on Proposition \ref{Ya1} and Proposition \ref{Ya2}, we can check the statement holds.
  \end{proof}

\subsection{Quadratic Rota-Baxter noncommutative pre-Poisson algebras and factorizable noncommutative pre-Poisson bialgebras}
In this section, we first introduce the notion of quadratic Rota-Baxter noncommutative pre-Poisson algebras. 
We then establish the relationship between factorizable noncommutative pre-Poisson bialgebras 
and these quadratic Rota-Baxter noncommutative pre-Poisson algebras.

\begin{defi} Assume that $(A,\succ,\prec,\ast,P)$ is a Rota-Baxter noncommutative pre-Poisson algebra of weight $\lambda$
and $(A,\succ,\prec,\ast,\omega)$ a quadratic noncommutative pre-Poisson algebra. Then $(A,\succ,\prec,\ast,P,\omega)$
is called a \textbf{quadratic Rota-Baxter noncommutative pre-Poisson algebra of weight $\lambda$} if the following condition holds:
\begin{equation} \label{Fs}\omega (P(x),y)+\omega(x, P(y))+\lambda\omega(x,y)=0, ~\forall~x, y \in A.\end{equation}
\end{defi}

\begin{defi} Let $(A,\cdot,[ \ , \ ],P)$ be a Rota-Baxter noncommutative Poisson algebra of weight $\lambda$ and
 $(A,\cdot,[ \ , \ ],\omega )$ be a symplectic noncommutative Poisson algebra if the following condition holds: 
 \begin{equation} \label{Fs1}\omega (P(x),y)+\omega(x, P(y))+\lambda\omega(x,y)=0, ~\forall~x, y \in A.\end{equation}
Then $(A,\cdot,[ \ , \ ],P,\omega)$ is called a \textbf{ symplectic Rota-Baxter noncommutative Poisson algebra of 
weight $\lambda$}.
\end{defi}

Observe that for all $x, y \in A$,
\begin{align*} &\lambda\omega(x,y)+\omega (-\lambda(x)- P(x),y)+\omega(x, -\lambda(y)- P(y))\\=&
-\lambda\omega(x,y)-\omega (P(x),y)-\omega(x,  P(y)), \end{align*}
from which we obtain the following conclusions.

\begin{pro} \label{Fb2} Let $(A,\succ,\prec,\ast,P,\omega)$ be a quadratic Rota-Baxter noncommutative pre-Poisson
 algebra and let $P:A\longrightarrow A$ be a linear
map. Then $(A,\succ,\prec,\ast,P,\omega)$ is a quadratic Rota-Baxter 
noncommutative pre-Poisson algebra of weight $\lambda$ if and only if
$(A,\succ,\prec,\ast,-\lambda I-P,\omega)$ is a quadratic Rota-Baxter noncommutative pre-Poisson algebra of weight $\lambda$.
\end{pro}

\begin{pro} Let $(A,\cdot,[ \ , \ ],\omega)$ be a quadratic noncommutative Poisson algebra and let $P:A\longrightarrow A$ be a linear
map. Then $(A,\cdot,[ \ , \ ] ,P,\omega)$ is a symplectic Rota-Baxter noncommutative Poisson algebra of weight $\lambda$ if and only if
$(A,\cdot,[ \ , \ ],-\lambda I-P,\omega)$ is a symplectic Rota-Baxter noncommutative Poisson algebra of weight $\lambda$.
\end{pro}

Based on Proposition \ref{Ps1}, the following can be obtained easily.

\begin{thm}\label{Fb0} Suppose that $(A,\cdot,[ \ , \ ],P,\omega)$ is a symplectic Rota-Baxter noncommutative Poisson algebra of 
weight $\lambda$. Then $(A,\succ,\prec,\ast,P,\omega)$ is a quadratic Rota-Baxter noncommutative pre-Poisson
algebra of weight $\lambda$, where $\succ,\prec,\ast$ are defined by Eqs.~(\ref{Ib1}) and (\ref{Ib2}) respectively.
 On the other hand, let $(A,\succ,\prec,\ast,P,\omega)$ be a quadratic Rota-Baxter noncommutative pre-Poisson algebra of weight $\lambda$.
Then $(A,\cdot,[ \ , \ ],P,\omega)$ is a symplectic Rota-Baxter noncommutative Poisson algebra
 of weight $\lambda$, where $x\cdot y=x\succ y+x\prec y,~[x,y]=x\ast y-y\ast x$.
\end{thm}

Let $\omega$ be a non-degenerate bilinear form on a vector space $A$. Then there is an isomorphism
$\omega^{\sharp}:A\longrightarrow A^{*}$ given by
\begin{equation} \omega(x,y)=\langle\omega^{\sharp}(x),y \rangle,~~\forall~x,y\in A.\end{equation}
Define an element $r_{\omega}\in A\otimes A$ such that $T_{r_{\omega}}=(\omega^{\sharp})^{-1}$,
that is, 
\begin{equation}\label{Nd1} \langle T_{r_{\omega}}(\zeta),\eta\rangle=\langle r_{\omega},\zeta \otimes\eta\rangle=
\langle (\omega^{\sharp})^{-1}(\zeta), \eta\rangle, \ \ \forall~\zeta,\eta\in A^{*}. \end{equation}

\begin{lem}\label{Fb1} Let $(A,\succ,\prec,\ast)$ be a noncommutative pre-Poisson algebra
 and $\omega$ be a non-degenerate bilinear form on $A$. Then
$(A, \succ,\prec,\ast,\omega)$ is a quadratic noncommutative pre-Poisson algebra if and only if the corresponding $r_{\omega}\in A\otimes A$ given by
Eq.~(\ref{Nd1}) is skew-symmetric and invariant.\end{lem}
\begin{proof}
It is obvious that $\omega$ is skew-symmetric if and only if $r_{\omega}$ is skew-symmetric. 
For all $x,y\in A$, put $\omega^{\sharp}(x)=\zeta,\omega^{\sharp}(y)=\eta,
\omega^{\sharp}(z)=\theta$ with $\zeta,\eta,\theta\in A^{*}$. If $\omega$ is skew-symmetric, we have
\begin{small}
\begin{align*} \omega (x \ast y, z)+\omega(y, [x, z])
=&\langle -\omega^{\sharp}(z), x\ast(\omega^{\sharp})^{-1}(\eta)\rangle
+\langle \omega^{\sharp}(y), [x, (\omega^{\sharp})^{-1}(\theta)]\rangle
\\=&-\langle \theta, x\ast T_{r_{\omega}}(\eta)\rangle
+\langle \eta, [x, T_{r_{\omega}}(\theta)]\rangle
\\=&\langle -R_{\ast}^{*}(T_{r_{\omega}}(\eta))\theta-\mathrm{ad}^{*}(T_{r_{\omega}}(\theta))\eta, x\rangle.\end{align*}
\end{small}
Thus, if $\omega$ is skew-symmetric, $\omega (x \ast y, z)=\omega(y, [x, z])$ if and only if
$R_{\ast}^{*}(T_{r_{\omega}}(\eta))\theta=-\mathrm{ad}^{*}(T_{r_{\omega}}(\theta))\eta$.
By the same token, if $\omega$ is skew-symmetric, $\omega (x \succ y, z)=\omega(y, z\cdot x)$ if and only if
$ L_{\cdot}^{*} (T_{r_{\omega}}(\zeta))\eta=-R_{\succ}^{*}(T_{r_{\omega}}(\eta))\zeta).$
Combining Proposition \ref{Si}, we complete the proof.
 \end{proof}
 
\begin{pro} \label{QF1} Let $(A, \succ,\prec,\ast,\omega)$ be a quadratic noncommutative pre-Poisson algebra and $r\in A\otimes A$. 
Assume that $r-\tau(r)$ is invariant. Define a linear map 
\begin{equation}\label{Nd2}P:A\longrightarrow A,\ \ \ P(x)=-T_{r}\omega^{\sharp}(x),~~\forall~x\in A.\end{equation}
Then $r$ is a solution of the NPP-YBE in $(A, \succ,\prec,\ast)$ if and only if $P$ satisfies 
\begin{small}
\begin{align}&\label{Nd3}P(x)\succ P(y)= P(P(x)\succ y+x\succ P(y)-x\succ T_{r-\tau(r)}\omega^{\sharp}(y)),\\&
\label{Nd4}P(x)\prec P(y)= P(P(x)\prec y+x\prec P(y)-x\prec T_{r-\tau(r)}\omega^{\sharp}(y))
\\&\label{Nd5}P(x)\ast P(y)= P(P(x)\ast y+x\ast P(y)-x\ast T_{r-\tau(r)}\omega^{\sharp}(y)),~~\forall~x,y\in A.
 \end{align}\end{small}
\end{pro}

\begin{proof} By Lemma \ref{Fb1}, $r_{\omega}$ is skew-symmetric and invariant. In light of Proposition \ref{Si},
 Eqs.~(\ref{Iv11})-(\ref{Iv13}) and (\ref{Apa2}),
 for all $x,y\in A$, put $\omega^{\sharp}(x)=\zeta,\omega^{\sharp}(y)=\eta$ with 
$\zeta,\eta\in A^{*}$, we get
\begin{small}
\begin{align*}& P(x)\succ P(y)=T_{r}\omega^{\sharp}(x)\succ T_{r}\omega^{\sharp}(y)=T_{r}(\zeta)\succ T_{r}(\eta),\\
&P(P(x)\succ y)=T_{r}\omega^{\sharp}(T_{r}(\zeta)\succ T_{r_{\omega}}(\eta))=
T_{r}\omega^{\sharp} T_{r_{\omega}}(R_{\cdot}^{*}(T_{r}(\zeta))\eta)
=T_{r}(R_{\cdot}^{*}(T_{r}(\zeta))\eta),
\\
&P(x\succ P(y))=T_{r}\omega^{\sharp}(T_{r_{\omega}}(\zeta)\succ T_{r}(\eta))=-
T_{r}\omega^{\sharp}T_{r_{\omega}}(L_{\prec}^{*}(T_{r}(\eta))\zeta)
=-T_{r}(L_{\prec}^{*}(T_{r}(\eta))\zeta),
\\
&P(x\succ T_{r-\tau(r)}\omega^{\sharp}(y))
=-T_{r}\omega^{\sharp}(T_{r_{\omega}}(\zeta)\succ T_{r-\tau(r)}(\eta))
=T_{r}\omega^{\sharp}  T_{r_{\omega}}( L_{\prec}^{*}(T_{r-\tau(r)}(\eta))\zeta)
\\=&T_{r}(( L_{\prec}^{*}(T_{r+\tau(r)}(\eta))\zeta)
=-T_{r}(\zeta\succ_{r-\tau(r)}\eta).\end{align*}\end{small}
 Thus, Eq.~(\ref{Nd3}) holds if and only if Eq.~(\ref{AD6}) holds.
By the same token, we can verify that Eq.~(\ref{Nd4})-(\ref{Nd5}) hold if and only if Eqs.~(\ref{AD5})-(\ref{AD6}) hold. 
Combining Theorem \ref{Aa}, we get the conclusion.

 \end{proof}
 
\begin{lem} \label{QF2} Let $A$ be a vector space and $\omega$ be a non-degenerate skew-symmetric bilinear form. 
Let $r\in A\otimes A$,
$\lambda\in k$ and a linear map $P$ given by Eq.~(\ref{Nd2}). Then $r$ satisfies
\begin{equation} \label{Nd5} r-\tau(r)=\lambda r_{\omega} \end{equation}
if and only if $P$ satisfies Eq.~(\ref{Fs}).
\end{lem}
\begin{proof} For all $x,y\in A$, put $\omega^{\sharp}(x)=\zeta,\omega^{\sharp}(y)=\eta$ with $\zeta,\eta\in A^{*}$.
\begin{small}
\begin{align*}& \omega(P(x),y)=-\omega(y,P(x))=\langle\omega^{\sharp}(y),T_{r}\omega^{\sharp}(x)\rangle
=\langle \eta,T_{r}(\zeta)\rangle
=\langle r,\zeta\otimes\eta\rangle,\\
&\omega(x,P(y))=-\langle \omega^{\sharp}(x),T_{r}\omega^{\sharp}(y)\rangle
=-\langle \zeta,T_{r}(\eta)\rangle
=-\langle r,\eta\otimes\zeta\rangle
=-\langle \tau(r),\zeta\otimes\eta\rangle,
\\
&\lambda\omega(x,y)=-\lambda \omega(y,x)
=-\lambda\langle \omega^{\sharp}(y),(\omega^{\sharp})^{-1}\omega^{\sharp}(x)\rangle
=-\lambda\langle r_{\omega},\zeta\otimes\eta\rangle
.\end{align*}\end{small}
Thus, Eq.~(\ref{Nd5}) holds if and only if Eq.~(\ref{Fs}) holds.
 \end{proof}
 
\begin{cor} \label{Fb3} Suppose that $(A, \succ,\prec,\ast,P,\omega)$ is a quadratic Rota-Baxter noncommutative pre-Poisson algebra of weight 0.
 Then there is a triangular
noncommutative pre-Poisson bialgebra $(A, \succ,\prec,\ast, \Delta_{\succ,r}, \Delta_{\prec,r},\delta_{r})$
 with $\Delta_{\succ,r}, \Delta_{\prec,r},\delta_{r}$ given
 by Eqs.~(\ref{CB1})-(\ref{CB2}) and (\ref{CB3}), where $r\in A\otimes A$ given by 
$T_{r}(\zeta)=P(\omega^{\sharp})^{-1}(\zeta)$ for all $\zeta\in A^{*}$.
\end{cor}

 \begin{proof} Since $r-\tau(r)=\lambda r_{\omega}=0$, $r$ is symmetric.
On the basis of Proposition \ref{QF1} and Lemma \ref{QF2}, we get the conclusion. \end{proof}
 
\begin{thm} \label{Fb3} Let $(A, \succ,\prec,\ast, \Delta_{\succ,r}, \Delta_{\prec,r},\delta_{r})$ be a factorizable 
noncommutative pre-Poisson bialgebra
with $r\in A\otimes A$. Then $(A,\succ,\prec,\ast, P,\omega)$
 is a quadratic Rota-Baxter noncommutative pre-Poisson algebra of weight $\lambda$ with $P$ given by Eq.~(\ref{Nd2}), and $\omega$ is given by
\begin{equation} \label{Nd6} \omega(x,y)=\lambda\langle T_{r-\tau(r)}^{-1}(x),y \rangle,~~\forall~x,y\in A.\end{equation}
Conversely, let $(A,\succ,\prec,\ast,P,\omega)$ be a quadratic Rota-Baxter noncommutative pre-Poisson algebra of weight $\lambda~(\lambda\neq 0)$.
Then there exists a factorizable noncommutative pre-Poisson bialgebra $(A, \succ,\prec,\ast,  \Delta_{\succ,r}, \Delta_{\prec,r},\delta_{r})$
 with $\Delta_{\succ,r}, \Delta_{\prec,r},\delta_{r}$ defined by Eqs.~(\ref{CB1})-(\ref{CB2}) and (\ref{CB3}) respectively, where $r\in A\otimes A$ is
given through the operator form $T_r=-P(\omega^{\sharp})^{-1}$.
\end{thm}

\begin{proof} On the one hand, since $(A, \succ,\prec,\ast, \Delta_{\succ,r}, \Delta_{\prec,r},\delta_{r})$ is a factorizable 
noncommutative pre-Poisson bialgebra, $r-\tau(r)$ is invariant and $T_{r-\tau(r)}$ is a linear isomorphism. 
By Proposition \ref{QF1} and Lemma \ref{QF2}, we get that $(A,\succ,\prec,\ast,P,\omega)$
 is a quadratic Rota-Baxter noncommutative pre-Poisson algebra of weight $\lambda$, where $\omega^{\sharp}=\lambda T_{r-\tau(r)}^{-1}$. 
 Conversely, assume that $(A,\succ,\prec,\ast,P,\omega)$ is a quadratic Rota-Baxter noncommutative 
 pre-Poisson algebra of weight $\lambda~(\lambda\neq 0)$.
 In light of Lemma \ref{Fb1},  Lemma \ref{QF2} and Proposition \ref{QF1}, 
 $r-\tau(r)$ is invariant, $T_{r-\tau(r)}=\lambda (\omega^{\sharp})^{-1}$ is a linear isomorphism
 and $r$ is a solution of the NPP-YBE in $(A, \succ,\prec,\ast)$. Thus,
  $(A, \succ,\prec,\ast, \Delta_{\succ,r}, \Delta_{\prec,r},\delta_{r})$ is a factorizable noncommutative pre-Poisson bialgebra.
\end{proof}

By Theorem \ref{Fb0} and Theorem \ref{Fb3}, we have the following correspondence:
\[(A, \cdot, [ \ , \ ], P, \omega){\overset{\text{Theorem 5.10}}{\longleftrightarrow}} 
(A, \succ,\prec,\ast, P, \omega){\overset{\text{Theorem 5.15}}{\longleftrightarrow}}(A, \succ,\prec,\ast, \Delta_{\succ,r}, \Delta_{\prec,r}, \delta_r).\]

 According to Proposition~\ref{Fb2} and Theorem~\ref{Fb3}, a quadratic Rota-Baxter noncommutative 
 pre-Poisson algebra $(A, \succ, \prec, \ast, P, \omega)$ of a non-zero 
 weight yields a factorizable noncommutative pre-Poisson bialgebra. Consequently, 
 $(A, \succ, \prec, \ast, -\lambda I - P, \omega)$ yields one as well. Furthermore, 
 Proposition~\ref{Qf} and Theorem~\ref{Fb3} imply that if 
 $(A, \succ, \prec, \ast, \Delta_{\succ, r}, \Delta_{\prec, r}, \delta_{r})$ is such a 
 factorizable bialgebra, then $(A, \succ, \prec, \ast, \Delta_{\succ, \tau(r)}, \Delta_{\prec, \tau(r)}, \delta_{\tau(r)})$ 
 also produces a quadratic Rota-Baxter noncommutative pre-Poisson algebra of weight $\lambda$. Indeed, we have

\begin{pro} Assume that $(A, \succ,\prec, \ast, \Delta_{\succ,r}, \Delta_{\prec,r},\delta_{r})$ is a factorizable 
noncommutative pre-Poisson bialgebra which corresponds to a quadratic Rota-Baxter noncommutative 
pre-Poisson algebra of non-zero weight $\lambda$. 
Then the factorizable 
noncommutative pre-Poisson bialgebra $(A,\succ,\prec, \ast, \Delta_{\succ,\tau(r)},\Delta_{\prec,\tau(r)},\delta_{\tau(r)})$ corresponds to the 
quadratic Rota-Baxter noncommutative pre-Poisson algebra $(A,\succ,\prec, \ast,-\lambda I-P,\omega)$ of non-zero weight $\lambda$. 
	
\end{pro}

\begin{proof} By
 Proposition \ref{Qf} and Theorem \ref{Fb1}, there is  a
quadratic Rota-Baxter noncommutative pre-Poisson algebra $(A,\succ,\prec, \ast,P^\prime,\omega^\prime)$ 
of non-zero weight $\lambda$
 corresponds to the factorizable noncommutative pre-Poisson bialgebra
  $(A, \succ,\prec, \ast, \Delta_{\succ,\tau(r)},\Delta_{\prec,\tau(r)},\delta_{\tau(r)})$.
Using Theorem \ref{Fb3},
\begin{equation*} 
\omega^{\prime}(x,y)=\lambda\langle T_{\tau(r)-r}^{-1}(x),y \rangle=-\omega(x,y).\end{equation*}
Using Eqs.~(\ref{Nd2}) and (\ref{Nd6}), 
\begin{small}
\begin{align*} 
P^{\prime}(x)=- T_{\tau(r)}{\omega^{\prime}}^{\sharp}(x)=T_{\tau(r)}\omega^{\sharp}(x)
&=\lambda T_{\tau(r)}T_{r-\tau(r)}^{-1}(x)=\lambda (T_{r}-T_{r-\tau(r)})T_{r-\tau(r)}^{-1}(x)
\\&=\lambda T_{r}T_{r-\tau(r)}^{-1}(x)-\lambda (x)
=T_{r}\omega^{\sharp}(x)-\lambda (x)=-P(x)-\lambda (x)
.\end{align*}\end{small}
Thus, the factorizable 
noncommutative pre-Poisson bialgebra $(A, \succ,\prec,\ast, \Delta_{\succ,\tau(r)},\Delta_{\prec,\tau(r)},\delta_{\tau(r)})$ yields a
quadratic Rota-Baxter noncommutative pre-Poisson algebra $(A,\succ,\prec, \ast,-\lambda I-P,-\omega)$ of non-zero weight $\lambda$.
Analogously, the converse part holds.
\end{proof}


\begin{center}{\textbf{Acknowledgments}}
\end{center}
This work was supported by the Natural Science
Foundation of Zhejiang Province of China (LY19A010001), the Science
and Technology Planning Project of Zhejiang Province
(2022C01118).

\begin{center} {\textbf{Statements and Declarations}}
\end{center}
 All datasets underlying the conclusions of the paper are available
to readers. No conflict of interest exits in the submission of this
manuscript.


\end {document}